\documentclass[a4paper,oneside,10pt]{article}%
\usepackage{amsmath}
\usepackage{amsfonts}
\usepackage{amssymb}
\usepackage{graphicx}
\usepackage{enumerate}
\usepackage{color}
\usepackage[square,numbers,sort&compress]{natbib}%
\usepackage{hyperref}
\usepackage[amsmath,thmmarks]{ntheorem}
\hypersetup{hypertex=true,bookmarks=true,bookmarksnumbered	=true,linktocpage=false,colorlinks=true,linkcolor=blue,anchorcolor=blue,citecolor=red}
\usepackage{indentfirst} 
\usepackage[mathscr]{euscript}
\providecommand{\U}[1]{\protect \rule{.1in}{.1in}}

\newtheorem{theorem}{Theorem}[section]

\newtheorem{definition}[theorem]{Definition}

\newtheorem{lemma}[theorem]{Lemma}

\newtheorem{remark}[theorem]{Remark}

\newenvironment{proof}[1][Proof]{\noindent \textbf{#1.} }{\  $\Box$}
\numberwithin{equation}{section}

\begin{document}
	
\makeatletter
\newcommand{\rmnum}[1]{\romannumeral #1}
\newcommand{\Rmnum}[1]{\expandafter\@slowromancap\romannumeral #1@}
\makeatother
{\theoremstyle{nonumberplain}
}

\title{Some properties of $G$-SVIEs}

\author{Renxing Li\textsuperscript{1,}\thanks{Corresponding author. E-mail address: 202011963@mail.sdu.edu.cn} 
	\and Xue Zhang\textsuperscript{2}  
}

\footnotetext[1]{Zhongtai Securities Institute for Financial Studies,
	Shandong University, Jinan, Shandong 250100, PR China. 202011963@mail.sdu.edu.cn.}
\footnotetext[2]{Department of Mathematics, National University of Defense Technology, Changsha, Hunan 410073, PR China.
	zhangxue\_1998@nudt.edu.cn.}

\maketitle
	
\textbf{Abstract}. In this paper, we investigated the solvability of $G$-SVIEs under two cases: time-varying Lipschitz coefficients and integral-Lipschitz coefficients. Using the Picard iteration method, we established the existence and uniqueness of solutions to $G$-SVIEs under these two conditions. Additionally, we prove the continuity of the solution with respect to parameters in parameter-dependent $G$-SVIEs with Lipschitz coefficients.

	
	
\addcontentsline{toc}{section}{\hspace*{1.8em}Abstract}

\section{Introduction}
In classical probability theory, stochastic Volterra integral equation (SVIE) was first studied in \cite{berger1980volterra,ito1979on}, and its form is as follows:
\begin{equation*}
	X(t)=\phi(t)+\int_{0}^{t} b(t, s, X(s)) ds +\int_{0}^{t} \sigma(t, s, X(s)) dB(s), \quad t\in [0,T].
\end{equation*} 
Different from the stochastic differential equation (SDE),  coefficients of the SVIE incorporate the current time $t$, enabling it to characterize stochastic systems with memory. Therefore, SVIE is a non-trivial generalization of SDE. The existence and uniqueness of solution to the SVIE were initiated by \cite{berger1980volterra,ito1979on}. Protter\cite{protter1985volterra} studied the SVIE driven by a semimartingale. \cite{wang2008existence} established the existence and uniqueness of solutions to SVIEs with singular kernels and non-Lipschitz coefficients. For the research on the solvability and the comparison theorem of SVIEs, one can refer to \cite{tudor1989a,ferreyra2000comparison,wang2010symmetrical,wang2015com,wang2017parameter,pardoux1990stochastic} and the references therein. In addition, SVIEs are also closely related to the field of stochastic control. For the work in this area, one can refer to \cite{yong2008well,shi2015optimal,yong2006backward,wang2018lq,chen2007a,wang2017optimal,zhang2010stochastic} and the references therein.

In recent years, to address uncertainty issues in financial markets, Peng\cite{peng2005nonlinear,peng2007G,peng2019nonlinear,peng2004filtration,peng2008multi,li2011stopping} introduced the $G$-expectation theory. Gao \cite{gao2009pathwise} investigated pathwise properties and homeomorphic property with respect to the initial values for $ G $-SDEs. Bai and Lin \cite{bai2014on} established the existence and uniqueness of solution to the $G$-SDE with integral-Lipschitz coefficients. Zhao\cite{zhao2025G-SVIE} introduced SVIEs driven by $G$-Brownian motion ($G$-SVIEs) and established the solvability under Lipschitz conditions. More related works can be found in \cite{liu2019multi,peng2019nonlinear,denis2011function,hu2016quasi} and the references therein. 

In this paper, we investigated the solvability of $G$-SVIEs under two cases: time-varying Lipschitz coefficients ((H1)-(H3) in Section \ref{sec3 G-SVIE time-varying}) and integral-Lipschitz coefficients ((H1')-(H3') in Section \ref{sec4 G-SVIE non-lip}). By means of Lemma \ref{sigma MG} and Lemma \ref{nonlip sigmaMG}, we verify that both types of coefficients belong to the $G$-expectation space. Then, using the Picard iteration method, we established the existence and uniqueness of solutions to $G$-SVIEs under these two conditions. Additionally, we prove the continuity of the solution with respect to parameters in parameter-dependent $G$-SVIEs with Lipschitz coefficients.

This paper is organized as follows. In Section \ref{sec2 pre}, we introduce some fundamental notations and results about $ G $-expectation theory. Section \ref{sec3 G-SVIE time-varying} established the existence and uniqueness of solutions to $G$-SVIEs with time-varying Lipschitz coefficients. The existence and uniqueness of solutions to $G$-SVIEs with
integral-Lipschitz coefficients is proved in Section \ref{sec4 G-SVIE non-lip}. Section \ref{sec5 G-SVIE para} studies $G$-SVIEs with a parameter.

\section{Preliminaries}\label{sec2 pre}

In this section, we recall some fundamental notions and results of the $G $-expectation and related $G $-stochastic analysis. More details can be found in \cite{peng2004filtration,peng2005nonlinear,peng2007G,peng2008multi,peng2019nonlinear}.

Let $\Omega = C_{0}(\mathbb{R}^+)$ denote the space of all $\mathbb{R}$-valued continuous functions $(\omega_t)_{t\in\mathbb{R}^+}$, with $\omega_0= 0$, equipped with the distance 
\begin{equation*}
	\rho(\omega^{1},\omega^{2}):=\sum_{i=1}^{\infty}2^{-i}[(\max_{t\in[0,i]}|\omega_t^{1}-\omega_t^{2}|)\wedge 1],\quad \omega^{1},\omega^{2}\in\Omega.
\end{equation*}
The corresponding canonical process is $B_t(\omega)=\omega_t$, $t\in\mathbb{R}^+ $. For each given $t\geq0$, we define 
\begin{equation*}
	Lip(\Omega_t):=\{\varphi(B_{t_1\wedge t},\cdots,B_{t_n\wedge t}):n\in\mathbb{N},t_1,\cdots,t_n\in[0,\infty),\varphi\in C_{l.Lip}(\mathbb{R}^{ n})\},\
	Lip(\Omega):=\cup_{n=1}^{\infty}Lip(\Omega_{n}),
\end{equation*}
where $C_{l,Lip}(\mathbb{R}^{ n}) $ is the linear space comprising all local Lipschitz functions on $\mathbb{R}^{n} $.

Let $G:\mathbb{R}\rightarrow\mathbb{R} $ be a given monotonic, sublinear
function, i.e., for $x\in\mathbb{R} $, 
\begin{align}  \label{G}
	G(x)=\frac{1}{2}(\bar{\sigma}^{2}x^{+}-\underline{\sigma}^{2}x^{-}),
\end{align}
where $0<\underline{\sigma}\leq\bar{\sigma}<\infty $. For each given
function $G$, Peng constructed the $G $-expectation $\hat{\mathbb{E}} $ on $(\Omega,Lip(\Omega)) $(see \cite{peng2019nonlinear} for definition). Then the canonical process $B $ is a one-dimensional $G $-Brownian motion under $\hat{\mathbb{E}} $.

For each $p\geq1 $, $L_{G}^{p}(\Omega) $ is defined as the completion of $Lip(\Omega) $ with respect to the norm $\|X\|_{L_{G}^{p}}=(\hat{\mathbb{E}}[|X|^{p}])^{\frac{1}{p}} $. Similarly, $L_{G}^{p}(\Omega_{t}) $ can be defined for each fixed $t $. $\hat{\mathbb{E}}:Lip(\Omega)\rightarrow\mathbb{R} $ can be continuously extended to the mapping from $L_{G}^{1}(\Omega) $ to $\mathbb{R} $. The sublinear expectation space $(\Omega,L_{G}^{1}(\Omega),\hat{\mathbb{E}}) $ is called $G $-expectation space.

Define $\mathcal{F}_{t}:=\sigma(B_{s}:s\leq t)$ and $\mathcal{F}:=\bigvee_{t\geq0}\mathcal{F}_{t} $. The representation theorem is stated as follows.

\begin{theorem}[\cite{hu2009representation,denis2011function}]
	\label{thm2.1} Let $(\Omega ,L_{G}^{1}(\Omega),\hat{\mathbb{E}}) $ be a $G $-expectation space. Then there exists a weakly compact set of probability measures $\mathcal{P} $ on $(\Omega,\mathcal{F}) $ such that 
	\begin{align}
		\hat{\mathbb{E}}[\xi]=\sup_{P\in\mathcal{P}}E_{P}[\xi], \ \text{ for each }\xi\in L_{G}^{1}(\Omega).
	\end{align}
\end{theorem}

From \cite{denis2011function}, we introduce the definition of capacity as
follows: 
\begin{equation*}
	\mathrm{c}(A):=\sup_{P\in\mathcal{P}}P(A), \ \text{ for each } A\in\mathcal{F}.
\end{equation*}

\begin{definition}
	A set $A\in\mathcal{F} $ is polar if $\mathrm{c}(A)=0 $ and a property is said to hold ``quasi-surely'\ (q.s.) if it holds outside a polar set.
\end{definition}

\begin{definition}
	For each $T>0 $ and $p\geq1 $, set 
	\begin{equation*}
		M_G^{p,0}(0,T):=\left\{\eta(t)=\sum_{j = 0}^{N - 1}\xi_jI_{[t_j,t_{j +1})}(t):N\in\mathbb{N},0= t_0< t_1<\cdots< t_N= T, \ \xi_j\in L_G^p(\Omega_{t_j})\right\}.
	\end{equation*}
	$M_{G}^{p}(0,T) $ is defined as the completion of $M_{G}^{p,0}(0,T) $ with respect to the norm $\|\eta\|_{M_{G}^{p}}:=(\hat{\mathbb{E}}
	[\int_{0}^{T}|\eta(t)|^{p}dt])^{\frac{1}{p}} $.
\end{definition}

According to \cite{li2011stopping,peng2007G,peng2008multi,peng2019nonlinear}, the integrals $\int_{0}^{t}\eta(s)dB_{s}$ and $\int_{0}^{t}\xi(s)d\langle
B\rangle_{s} $ are well-defined for $\eta(\cdot)\in M_{G}^{2}(0,T) $ and $\xi(\cdot)\in M_{G}^{1}(0,T) $, where $\langle B\rangle $ denotes the
quadratic variation process of $B $. By Proposition 3.4.5 and Corollary
3.5.5 in \cite{peng2019nonlinear}, we have 
\begin{equation*}
	\underline{\sigma}^{2}\mathrm{d}t\leq \mathrm{d}\langle B\rangle(t)\leq\bar{\sigma}^{2}\mathrm{d}t,\quad \text{q.s.}
\end{equation*}
and for $\eta(\cdot)\in M_{G}^{2}(0,T) $, 
\begin{align}  \label{ito}
	\hat{\mathbb{E}}\left[ \left|\int_{0}^{t}\eta(s)\mathrm{d}B(s) \right|^{2} \right] =\hat{\mathbb{E}}\left[\int_{0}^{t}|\eta(s)|^{2}\mathrm{d}\langle B\rangle(s)\right]\leq\bar{\sigma}^{2} \hat{\mathbb{E}}\left[\int_{0}^{t}|\eta(s)|^{2}\mathrm{d}s\right]
	.
\end{align}

\begin{theorem}[\cite{hu2014backward}]\label{BDG}
	Let $p\geq2 $ and $T>0 $. For all $\eta(\cdot)\in M_{G}^{p}(0,T) $, we have 
	\begin{align}  \label{bdg}
		\underline{\sigma}^{p}c_{p}\hat{\mathbb{E}}\left[\left(\int_{0}^{T}|\eta(s)|^{2}\mathrm{d}s\right)^{\frac{p}{2}}\right]\leq\hat{\mathbb{E}}\left[\sup_{t\in[0,T]}\left|\int_{0}^{t}\eta(s)\mathrm{d}B_{s} \right|^{p} \right] \leq\bar{\sigma}^{p}C_{p} \hat{\mathbb{E}}\left[\left(\int_{0}^{T}|\eta(s)|^{2}\mathrm{d}s\right)^{\frac{p}{2}} \right],
	\end{align}
	where $0<c_{p}<C_{p}<\infty $ are constants.
\end{theorem}

For each $T>0 $ and $p\geq1 $, define 
\begin{equation*}
	\tilde{M}_{G}^{p}(0,T):=\left\lbrace X(\cdot)\in M_{G}^{p}(0,T): X(t) \in L_{G}^{p}(\Omega_{t}) \text{ for each } t\in[0,T] \right\rbrace .
\end{equation*}

\begin{definition}
	A process $X(\cdot )\in \tilde{M}_{G}^{2}(0,T)$ is said to be mean-square continuous if for any $t^{\prime },t\in \lbrack 0,T]$, 
	\begin{equation*}
		\lim_{t^{\prime }\rightarrow t}\hat{\mathbb{E}}\left[ \left\vert
		X(t)-X(t^{\prime })\right\vert ^{2}\right] =0.
	\end{equation*}
\end{definition}

\begin{theorem}[\protect\cite{denis2011function}]
	\label{continuous} Let $X(\cdot)\in \tilde{M}_{G}^{p}(0,T) $ for $T>0 $ and $p\geq1 $. If there exists positive constants $c $ and $\varepsilon $ such that 
	\begin{equation*}
		\mathbb{E}[|X(t) - X(s)|^p]\leq c|t - s|^{1+\varepsilon},
	\end{equation*}
	then there exists a modification $\tilde{X}(\cdot) $ of $X(\cdot) $, which satisfies 
	\begin{equation*}
		\mathbb{E}\left[\left(\sup_{s\neq t}\frac{|\tilde{X}(t)-\tilde{X}(s)|}{|t-s|^{\alpha}}\right)^p\right]<\infty,
	\end{equation*}
	for every $\alpha\in[0,\varepsilon/p) $. Consequently, paths of $\tilde{X}(\cdot) $ are $\alpha $-order quasi-surely H\"{o}lder continuous for every $\alpha\in[0,\varepsilon/p) $. Moreover, $\tilde{X}(\cdot)\in \tilde{M}_{G}^{p}(0,T) $.
\end{theorem}

\begin{lemma}[Bihari's inequality]\label{Bihari's inequality} 
	Let $\gamma :\mathbb{R}^{+}\rightarrow \mathbb{R}^{+}$ be a continuous and increasing function satisfies
	\begin{equation*}
		\int_{0^{+}}\frac{\mathrm{d}s}{\gamma \left( s\right) }=+\infty ,
		\quad \psi \left( 0\right) =0.
	\end{equation*}%
	If $u$ is a non-negative function on $\mathbb{R}^{+}$ such that
	\begin{equation*}
		u(t)\leq \int_{0}^{t}\gamma \left( u(s)\right) \mathrm{d}s,\quad t\geq 0,
	\end{equation*}
	then $u(t),t\geq 0$.
\end{lemma}

\begin{lemma}[Jensen's inequality]
	\label{Jensen's inequality}Let $\rho :\mathbb{R}\rightarrow\mathbb{R}$ be a concave and increasing function. Then for $\xi \in L_{G}^{1}(\Omega )$, 
	\begin{equation*}
		\hat{\mathbb{E}}[\rho \left( \xi \right) ]\leq \rho \left( \hat{\mathbb{E}}[\xi ]\right) .
	\end{equation*}
\end{lemma}

\section{$G$-SVIE with time-varying coefficients}\label{sec3 G-SVIE time-varying}

Let $(\Omega ,L_{G}^{1}(\Omega ),\hat{\mathbb{E}})$ be a $G$-expectation
space and let $0<\underline{\sigma }\leq \bar{\sigma}<\infty $. We shall
consider the following $G$-stochastic Volterra integral equation ($G$-SVIE): for each given $0\leq T<\infty $, 
\begin{equation}
	X(t)=\phi(t)+\int_{0}^{t}b(t,s,X(s))\mathrm{d}s+\int_{0}^{t}h(t,s,X(s))\mathrm{d}\langle B\rangle (s)+\int_{0}^{t}\sigma (t,s,X(s))\mathrm{d}B(s),\ t\in \lbrack 0,T].  \label{G-SVIE}
\end{equation}
Define $\bigtriangleup (t,s)=\left\{ (t,s)\in \lbrack 0,T]^{2}:s\leq
t\right\} $. Here $b(\omega,t,s,x),h(\omega,t,s,x),\sigma(\omega,t,s,x):\Omega\times\bigtriangleup(t,s)\times\mathbb{R}\rightarrow \mathbb{R}$ and $\phi (\omega ,t):\Omega \times \lbrack 0,T]\rightarrow \mathbb{R}$. The letter $C$ is constant which depends on the subscripts and may differ in different proofs. In the following assumptions, $\epsilon $ and $\bar{\epsilon}$ are constants and satisfy $0\leq \epsilon <\bar{\epsilon}$.

\begin{description}
	\item[(H1)] For each $t\in \lbrack 0,T]$ and $x\in \mathbb{R},\ b(t,\cdot,x),\ h(t,\cdot ,x),\ \sigma (t,\cdot ,x)\in M_{G}^{2+\epsilon }(0,t)$.
	
	\item[(H2)] For all $x,y\in 
	\mathbb{R}$ and $(t,s)\in \bigtriangleup (t,s)$, there exists a deterministic function $L(t,s):$ $\bigtriangleup (t,s)\rightarrow \mathbb{R}^{+}$ such that 
	\begin{align*}
		&\left\vert b(t,s,x)-b(t,s,y)\right\vert +\left\vert
		h(t,s,x)-h(t,s,y)\right\vert +\left\vert \sigma (t,s,x)-\sigma
		(t,s,y)\right\vert \leq L(t,s)\left\vert x-y\right\vert , \\
		&\left\vert b(t,s,x)\right\vert ^{2+\epsilon }+\left\vert h(t,s,x)\right\vert^{2+\epsilon }+\left\vert \sigma (t,s,x)\right\vert ^{2+\epsilon } \leq L^{2+\epsilon }(t,s)\left( 1+\left\vert x\right\vert ^{2+\epsilon }\right) .
	\end{align*}
	Meanwhile, $\sup_{t\in \lbrack 0,T]}\int_{0}^{T}L^{2+\bar{\epsilon}}(t,s)%
	\mathrm{d}s<\infty $.
	
	\item[(H3)] Let $t_{1}\geq t_{2}$. For all $x\in \mathbb{R}$ and $(t_{1},s)\in\bigtriangleup(t_{1},s),(t_{2},s)\in\bigtriangleup(t_{2},s)$, there exists a deterministic positive function $K(t_{1},t_{2},s)$ such that 
	\begin{align*}
		&\left\vert b(t_{1},s,x)-b(t_{2},s,x)\right\vert ^{2+\epsilon }+\left\vert h(t_{1},s,x)-h(t_{2},s,x)\right\vert ^{2+\epsilon }+\left\vert \sigma(t_{1},s,x)-\sigma (t_{2},s,x)\right\vert ^{2+\epsilon }\\
		\leq & K^{2+\epsilon}(t_{1},t_{2},s)\left( 1+\left\vert x\right\vert ^{2+\epsilon }\right) .
	\end{align*}
	Here $K(t_{1},t_{2},s)$ satisfies $\int_{0}^{t_{2}}K^{2+\epsilon}(t_{1},t_{2},s)\mathrm{d}s\leq C_{T}\rho \left( \left\vert t_{1}-t_{2}\right\vert \right) $ where $\rho :\mathbb{R}^{+}\rightarrow \mathbb{R}^{+}$ is continuous and strictly increasing with $\rho (0)=0$.
\end{description}

To ensure that the stochastic integral is well-defined in the $G$-expectation space, the following lemma is required.

\begin{lemma}\label{sigma MG}
	Let $0\leq \epsilon <\bar{\epsilon}$. Define  \[M(t):=\int_{0}^{t}\sigma (t,s,X(s))\mathrm{d}B(s), \ N(t):=\int_{0}^{t}b(t,s,X(s))\mathrm{d}s+\int_{0}^{t}h(t,s,X(s))\mathrm{d}\langle B\rangle (s).\] 
	Assume that $X(\cdot )\in M_{G}^{2+\epsilon }(0,T)$ and $\sup_{t\in\lbrack0,T]}\hat{\mathbb{E}}\left[\left\vert X(t)\right\vert ^{2+\epsilon }\right] <\infty $. If (H1)-(H3) hold, then $M(\cdot )+N(\cdot )\in \tilde{M}_{G}^{2+\epsilon }(0,T)$.
\end{lemma}

\begin{proof}
	For each $n>0$, define 
	\begin{equation*}
		\sigma _{n}(t,s,x)=\sigma (t,s,x)I_{\left\{ L(t,s)\leq n\right\}}.
	\end{equation*}
	Clearly, for a fixed $t$, $\sigma _{n}(t,s,x)$ is Lipschitz continuous with respect to $x$. According to Theorem 4.7 in \cite{hu2016quasi}, we get $\sigma_{n}(t,\cdot ,x)\in M_{G}^{2+\epsilon }(0,t)$. Thanks to Lemma 3.3 in \cite{liu2019multi}, we yield $\sigma _{n}(t,\cdot ,X(\cdot ))\in M_{G}^{2+\epsilon}(0,t) $. From (H2), 
	\begin{align*}
		\hat{\mathbb{E}}\left[ \int_{0}^{t}\left\vert \sigma _{n}(t,s,X(s))-\sigma(t,s,X(s))\right\vert ^{2+\epsilon }\mathrm{d}s\right] =&\hat{\mathbb{E}}\left[ \int_{0}^{t}\left\vert \sigma (t,s,X(s))\right\vert ^{2+\epsilon}I_{\left\{ L(t,s)>n\right\} }\mathrm{d}s\right] \\
		\leq &\hat{\mathbb{E}}\left[ \int_{0}^{t}L^{2+\epsilon }(t,s)\left(1+\left\vert X(s)\right\vert ^{2+\epsilon }\right) I_{\left\{L(t,s)>n\right\} }\mathrm{d}s\right] \\
		\leq &\int_{0}^{t}L^{2+\epsilon }(t,s)\hat{\mathbb{E}}\left[ 1+\left\vert X(s)\right\vert ^{2+\epsilon }\right] I_{\left\{ L(t,s)>n\right\} }\mathrm{d}s \\
		\leq &\left( \sup_{t\in \lbrack 0,T]}\hat{\mathbb{E}}\left[ 1+\left\vert X(t)\right\vert ^{2+\epsilon }\right] \right) \int_{0}^{t}L^{2+\epsilon}(t,s)I_{\left\{ L(t,s)>n\right\} }\mathrm{d}s \\
		\leq &\left( \sup_{t\in \lbrack 0,T]}\hat{\mathbb{E}}\left[ 1+\left\vert X(t)\right\vert ^{2+\epsilon }\right] \right) \int_{0}^{t}\frac{L^{2+\bar{\epsilon}}(t,s)}{n^{\bar{\epsilon}-\epsilon }}\mathrm{d}s.
	\end{align*}%
	By (H2), we have 
	\begin{align*}
		&\lim_{n\rightarrow \infty }\hat{\mathbb{E}}\left[ \int_{0}^{t}\left\vert\sigma _{n}(t,s,X(s))-\sigma (t,s,X(s))\right\vert ^{2+\epsilon }\mathrm{d}s
		\right] \\
		\leq &\lim_{n\rightarrow \infty }\left( \sup_{t\in \lbrack 0,T]}\hat{\mathbb{E}}\left[ 1+\left\vert X(t)\right\vert ^{2+\epsilon}\right]\right)\int_{0}^{t}\frac{L^{2+\bar{\epsilon}}(t,s)}{n^{\bar{\epsilon}-\epsilon }}\mathrm{d}s \\
		\leq&\lim_{n\rightarrow\infty}\frac{1}{n^{\bar{\epsilon}-\epsilon }}\left( \sup_{t\in \lbrack 0,T]}\hat{\mathbb{E}}\left[ 1+\left\vert X(t)\right\vert ^{2+\epsilon }\right] \right) \left( \sup_{t\in\lbrack0,T]}\int_{0}^{T}L^{2+\bar{\epsilon}}(t,s)\mathrm{d}s\right) \\
		=&0.
	\end{align*}
	Then we obtain $\sigma (t,\cdot ,X(\cdot ))\in M_{G}^{2+\epsilon }(0,t)$, which implies $M(t)\in L_{G}^{2+\epsilon }(\Omega _{t})$.
	
	We now prove that $M(\cdot )\in M_{G}^{2+\epsilon }(0,T)$. Denote the partition of $[0,T]$ by $\Pi_{T}^{m}=\left\{0=t_{0}^{m}<...<t_{m}^{m}=T\right\} $, and let
	\begin{equation*}
		\left\Vert \Pi _{T}^{m}\right\Vert =\max \left\{ t_{i+1}^{m}-t_{i}^{m},0\leq i\leq m-1\right\} .
	\end{equation*}
	Then $\left\Vert \Pi _{T}^{m}\right\Vert \rightarrow 0$ as $m\rightarrow \infty $. Set 
	\begin{equation*}
		M_{m}(t)=\sum_{i=0}^{m-1}M(t_{i}^{m})I_{[t_{i}^{m},t_{i+1}^{m})}(t).
	\end{equation*}
	It is clear that $M_{m}(\cdot )\in M_{G}^{2}(0,T)$. For each $t\in \lbrack t_{i}^{m},t_{i+1}^{m})$, by H\"{o}lder's inequality, it follows that%
	\begin{align}
		&\hat{\mathbb{E}}\left[ \left\vert M(t)-M_{m}(t_{i}^{m})\right\vert^{2+\epsilon }\right]  \notag \\
		\leq &\hat{\mathbb{E}}\left[ \left\vert\int_{0}^{t_{i}^{m}}\sigma
		(t,s,X(s))-\sigma(t_{i}^{m},s,X(s))\mathrm{d}B_{s}+\int_{t_{i}^{m}}^{t}\sigma (t,s,X(s))\mathrm{d}B_{s}\right\vert ^{2+\epsilon }\right]  \notag \\
		\leq &C_{\bar{\sigma},\epsilon }\left\{ \hat{\mathbb{E}}\left[ \left\vert\int_{0}^{t_{i}^{m}}\left\vert \sigma (t,s,X(s))-\sigma
		(t_{i}^{m},s,X(s))\right\vert^{2}\mathrm{d}s\right\vert^{\frac{2+\epsilon }{2}}\right] +\hat{\mathbb{E}}\left[ \left\vert\int_{t_{i}^{m}}^{t}\left\vert \sigma (t,s,X(s))\right\vert ^{2}\mathrm{d}s\right\vert ^{\frac{2+\epsilon }{2}}\right] \right\}  \notag \\
		\leq &C_{\bar{\sigma},T,\epsilon }\left\{ \hat{\mathbb{E}}\left[\int_{0}^{t_{i}^{m}}\left\vert \sigma (t,s,X(s))-\sigma(t_{i}^{m},s,X(s))\right\vert ^{2+\epsilon }\mathrm{d}s\right] +\hat{\mathbb{E}}\left[ \int_{t_{i}^{m}}^{t}\left\vert \sigma (t,s,X(s))\right\vert
		^{2+\epsilon }\mathrm{d}s\right] \left\vert t-t_{i}^{m}\right\vert ^{\frac{\epsilon }{2}}\right\}  \notag \\
		\leq &C_{\bar{\sigma},T,\epsilon }\left\{ \hat{\mathbb{E}}\left[
		\int_{0}^{t_{i}^{m}}K^{2+\epsilon }(t,t_{i}^{m},s)\left( 1+\left\vert X(s)\right\vert ^{2+\epsilon }\right) \mathrm{d}s\right]+\hat{\mathbb{E}}
		\left[ \int_{t_{i}^{m}}^{t}L^{2+\epsilon }(t,s)\left( 1+\left\vert X(s)\right\vert ^{2+\epsilon }\right) \mathrm{d}s\right] \left\vert t-t_{i}^{m}\right\vert ^{\frac{\epsilon }{2}}\right\}  \notag \\
		\leq &C_{\bar{\sigma},T,\epsilon }\left\{ \int_{0}^{t_{i}^{m}}K^{2+\epsilon}(t,t_{i}^{m},s)\hat{\mathbb{E}}\left[ 1+\left\vert X(s)\right\vert^{2+\epsilon }\right] \mathrm{d}s+\left\vert t-t_{i}^{m}\right\vert ^{\frac{\epsilon }{2}}\int_{t_{i}^{m}}^{t}L^{2+\epsilon }(t,s)\hat{\mathbb{E}}\left[1+\left\vert X(s)\right\vert ^{2+\epsilon }\right] \mathrm{d}s\right\} \notag \\
		\leq &C_{\bar{\sigma},T,\epsilon }\left\{ \left( \sup_{t\in \lbrack 0,T]}\hat{\mathbb{E}}\left[ 1+\left\vert X(t)\right\vert ^{2+\epsilon }\right]\right) \int_{0}^{t_{i}^{m}}K^{2+\epsilon }(t,t_{i}^{m},s)\mathrm{d}s\right. \notag\\
		&\left. +\left\vert t-t_{i}^{m}\right\vert ^{\frac{\epsilon}{2}}\left(\int_{t_{i}^{m}}^{t}L^{2+\bar{\epsilon}}(t,s)\mathrm{d}s\right) ^{\frac{2+\epsilon }{2+\bar{\epsilon}}}\left( \int_{t_{i}^{m}}^{t}\left( \hat{\mathbb{E}}\left[ 1+\left\vert X(s)\right\vert ^{2+\epsilon }\right]\right)^{\frac{2+\bar{\epsilon}}{\bar{\epsilon}-\epsilon }}\mathrm{d}s\right) ^{\frac{\bar{\epsilon}-\epsilon }{2+\bar{\epsilon}}}\right\}  \notag \\
		\leq &C_{\bar{\sigma},T,\epsilon }\left( \sup_{t\in \lbrack 0,T]}\hat{\mathbb{E}}\left[ 1+\left\vert X(t)\right\vert ^{2+\epsilon }\right] \right)\left\{ C_{T}\rho \left( \left\vert t-t_{i}^{m}\right\vert\right)+\left(\int_{t_{i}^{m}}^{t}L^{2+\bar{\epsilon}}(t,s)\mathrm{d}s\right) ^{\frac{2+\epsilon }{2+\bar{\epsilon}}}\left\vert t-t_{i}^{m}\right\vert ^{\frac{\bar{\epsilon}\left( 2+\epsilon \right) }{2\left( 2+\bar{\epsilon}\right) }}\right\}.  \label{sigam t t'}
	\end{align}
	Thus, 
	\begin{align*}
		&\sup_{t\in \lbrack 0,T]}\hat{\mathbb{E}}\left[ \left\vert M(t)-M_{m}(t)\right\vert ^{2+\epsilon }\right] \\
		\leq&\sup_{t\in\lbrack0,T]}\sum_{i=0}^{m-1}\hat{\mathbb{E}}\left[
		\left\vert M(t)-M_{m}(t_{i}^{m})\right\vert ^{2+\epsilon }\right]
		I_{[t_{i}^{m},t_{i+1}^{m})}(t) \\
		\leq &C_{\bar{\sigma},T,\epsilon }\left( \sup_{t\in \lbrack 0,T]}\hat{\mathbb{E}}\left[ 1+\left\vert X(t)\right\vert ^{2+\epsilon }\right] \right)\left\{ C_{T}\rho \left( \left\Vert \Pi _{T}^{m}\right\Vert \right) +\left(\sup_{t\in \lbrack 0,T]}\int_{0}^{T}L^{2+\bar{\epsilon}}(t,s)\mathrm{d}s\right) ^{\frac{2+\epsilon }{2+\bar{\epsilon}}}\left\Vert \Pi_{T}^{m}\right\Vert ^{\frac{\bar{\epsilon}\left( 2+\epsilon \right) }{2\left( 2+\bar{\epsilon}\right) }}\right\} .
	\end{align*}
	It leads to 
	\begin{align*}
		&\lim_{m\rightarrow \infty }\hat{\mathbb{E}}\left[ \int_{0}^{T}\left\vert M(t)-M_{m}(t)\right\vert ^{2+\epsilon }\mathrm{d}t\right] \\
		\leq &\lim_{m\rightarrow \infty }T\sup_{t\in \lbrack 0,T]}\hat{\mathbb{E}}\left[ \left\vert M(t)-M_{m}(t)\right\vert ^{2+\epsilon }\right] \\
		\leq &\lim_{m\rightarrow \infty }C_{\bar{\sigma},T,\epsilon }\left(\sup_{t\in \lbrack 0,T]}\hat{\mathbb{E}}\left[ 1+\left\vert X(t)\right\vert^{2+\epsilon }\right] \right) \left\{ C_{T}\rho \left( \left\Vert \Pi_{T}^{m}\right\Vert \right)+\left(\sup_{t\in\lbrack0,T]}\int_{0}^{T}L^{2+\bar{\epsilon}}(t,s)\mathrm{d}s\right)^{\frac{2+\epsilon}{2+\bar{\epsilon}}}\left\Vert \Pi _{T}^{m}\right\Vert ^{\frac{\bar{\epsilon}\left( 2+\epsilon\right) }{2\left( 2+\bar{\epsilon}\right) }}\right\} \\
		=&0.
	\end{align*}%
	Then we obtain $M(\cdot )\in M_{G}^{2+\epsilon }(0,T)$. Furthermore, $M(\cdot)\in \tilde{M}_{G}^{2+\epsilon }(0,T)$. Similarly, we can prove that $N(\cdot )\in \tilde{M}_{G}^{2+\epsilon }(0,T)$.
\end{proof}

The following theorem is the main result of this section.

\begin{theorem}\label{solution} 
	Let $0=\epsilon <\bar{\epsilon}$. Assume that $\phi (\cdot)\in M_{G}^{2}(0,T)$ satisfies 
	\begin{equation*}
		\sup_{t\in \lbrack 0,T]}\hat{\mathbb{E}}[|\phi (t)|^{2}]<\infty .
	\end{equation*}%
	If (H1)-(H3) hold, then there exists a unique solution $X(\cdot )\in
	M_{G}^{2}(0,T)$ of (\ref{G-SVIE}) satisfies $\sup\limits_{t\in \lbrack 0,T]}%
	\hat{\mathbb{E}}[|X(t)|^{2}]<\infty $. Moreover, $X(\cdot )-\phi (\cdot )\in 
	\tilde{M}_{G}^{2}(0,T)$ and is mean-square continuous.
\end{theorem}

\begin{proof}[Proof]
	Existence. Let $X_{0}(t)=\phi (t)$ and let 
	\begin{equation}
		X_{n}(t)=\phi(t)+\int_{0}^{t}b(t,s,X_{n-1}(s))\mathrm{d}s+\int_{0}^{t}h(t,s,X_{n-1}(s))\mathrm{d}\left\langle B\right\rangle(s)+\int_{0}^{t}\sigma(t,s,X_{n-1}(s))\mathrm{d}B(s)  \label{Xn}
	\end{equation}
	where $n=1,2,\cdots $. If $X_{n-1}(\cdot )\in M_{G}^{2}[0,T]$ and $%
	\sup\limits_{t\in\lbrack0,T]}\hat{\mathbb{E}}[|X_{n-1}(t)|^{2}]<\infty $, then by Theorem \ref{BDG}, we have 
	\begin{align}
		\hat{\mathbb{E}}[|X_{n}(t)|^{2}]=& \hat{\mathbb{E}}\left[ \left\vert\phi(t)+\int_{0}^{t}b(t,s,X_{n-1}(s))\mathrm{d}s+\int_{0}^{t}h(t,s,X_{n-1}(s))\mathrm{d}\langle B\rangle (s)+\int_{0}^{t}\sigma (t,s,X_{n-1}(s))\mathrm{d}B(s)\right\vert ^{2}\right]   \notag \\
		\leq&C\left(\hat{\mathbb{E}}[|\phi(t)|^{2}]+\hat{\mathbb{E}}\left[\left\vert \int_{0}^{t}b(t,s,X_{n-1}(s))\mathrm{d}s\right\vert ^{2}\right] +\hat{\mathbb{E}}\left[ \left\vert \int_{0}^{t}h(t,s,X_{n-1}(s))\mathrm{d}\langle B\rangle (s)\right\vert ^{2}\right] \right.   \notag \\
		& \left. +\hat{\mathbb{E}}\left[ \left\vert \int_{0}^{t}\sigma
		(t,s,X_{n-1}(s))\mathrm{d}B(s)\right\vert ^{2}\right] \right)   \notag \\
		\leq & C\left( \hat{\mathbb{E}}[|\phi (t)|^{2}]+\left( (\bar{\sigma}^{4}+1)T+\bar{\sigma}^{2}\right) \hat{\mathbb{E}}\left[\int_{0}^{t}L^{2}(t,s)(1+|X_{n-1}(s)|^{2})\mathrm{d}s\right] \right)   \notag\\
		\leq & C\left( \sup_{t\in \lbrack 0,T]}\hat{\mathbb{E}}[|\phi
		(t)|^{2}]+\left( (\bar{\sigma}^{4}+1)T+\bar{\sigma}^{2}\right) \left(\sup_{t\in\lbrack0,T]}\int_{0}^{T}L^{2}(t,s)\mathrm{d}s\right) \left(\sup_{t\in \lbrack 0,T]}\hat{\mathbb{E}}[1+|X_{n-1}(t)|^{2}]\right) \right) 
		\notag \\
		<& \infty .  \label{supXn}
	\end{align}
	Thus $X_{n}(\cdot )\in M_{G}^{2}[0,T]$ and $\sup\limits_{t\in \lbrack 0,T]}\hat{\mathbb{E}}[|X_{n}(t)|^{2}]<\infty $. Since $X_{0}(\cdot )\in M_{G}^{2}[0,T]$ and $\sup\limits_{t\in \lbrack 0,T]}\hat{\mathbb{E}}[|X_{0}(t)|^{2}]<\infty $, it follows that $X_{n}(\cdot )\in M_{G}^{2}[0,T]$ and $\sup\limits_{t\in \lbrack 0,T]}\hat{\mathbb{E}}[|x_{n}(t)|^{2}]<\infty $
	for $n=1,2,\cdots $. Denote 
	\begin{equation*}
		\hat{b}(x-y):=b(t,s,x)-b(t,s,y),\ \hat{h}(x-y):=h(t,s,x)-h(t,s,y),\ \hat{\sigma}(x-y):=\sigma (t,s,x)-\sigma (t,s,y).
	\end{equation*}
	For convenience, set $\left( \sup_{t\in \lbrack 0,T]}\int_{0}^{T}L^{2+\bar{\epsilon}}(t,s)\mathrm{d}s\right) ^{\frac{2}{2+\bar{\epsilon}}}=L_{\bar{\epsilon}}$. Therefore, 
	\begin{align*}
		&\hat{\mathbb{E}}[|X_{n+1}(t)-X_{n}(t)|^{2}]\\
		=& \hat{\mathbb{E}}\left[\left\vert\int_{0}^{t}\hat{b}(X_{n}(s)-X_{n-1}(s))\mathrm{d}s+\int_{0}^{t}\hat{h}(X_{n}(s)-X_{n-1}(s))\mathrm{d}\langle B\rangle(s)\int_{0}^{t}\hat{\sigma}(X_{n}(s)-X_{n-1}(s))\mathrm{d}B(s)\right\vert ^{2}\right]  \\
		\leq & C\left( \hat{\mathbb{E}}\left[ \left\vert \int_{0}^{t}\hat{b}(X_{n}(s)-X_{n-1}(s))\mathrm{d}s\right\vert ^{2}\right] +\hat{\mathbb{E}}
		\left[\left\vert\int_{0}^{t}\hat{h}(X_{n}(s)-X_{n-1}(s))\mathrm{d}\langle B\rangle (s)\right\vert ^{2}\right] \right.  \\
		& +\left. \hat{\mathbb{E}}\left[ \left\vert \int_{0}^{t}\hat{\sigma}(X_{n}(s)-X_{n-1}(s))\mathrm{d}B(s)\right\vert ^{2}\right] \right)  \\
		\leq & C\left( (\bar{\sigma}^{4}+1)T+\bar{\sigma}^{2}\right)
		\int_{0}^{t}L^{2}(t,s)\hat{\mathbb{E}}[|X_{n}(s)-X_{n-1}(s)|^{2}]\mathrm{d}s\\
		\leq& C_{\bar{\sigma},T}L_{\bar{\epsilon}}\left( \int_{0}^{t}\left(\hat{\mathbb{E}}[|X_{n}(s)-X_{n-1}(s)|^{2}]\right) ^{\frac{2+\bar{\epsilon}}{\bar{\epsilon}}}\mathrm{d}s\right) ^{\frac{\bar{\epsilon}}{2+\bar{\epsilon}}}.
	\end{align*}
	Note that 
	\begin{align*}
		\hat{\mathbb{E}}[|X_{1}(t)-X_{0}(t)|^{2}]=&\hat{\mathbb{E}}\left[
		\left\vert\int_{0}^{t}b(t,s,X_{0}(s))\mathrm{d}s+\int_{0}^{t}h(t,s,X_{0}(s))\mathrm{d}\langle B\rangle (s)+\int_{0}^{t}\sigma (t,s,X_{0}(s))\mathrm{d}B(s)\right\vert ^{2}\right]  \\
		\leq&C\left((\bar{\sigma}^{4}+1)T+\bar{\sigma}^{2}\right)\int_{0}^{t}L^{2}(t,s)\hat{\mathbb{E}}[1+|\phi (s)|^{2}]\mathrm{d}s \\
		\leq & C_{\bar{\sigma},T}L_{\bar{\epsilon}}\left( \int_{0}^{t}(\hat{\mathbb{E}}[1+|\phi(s)|^{2}])^{\frac{2+\bar{\epsilon}}{\bar{\epsilon}}}\mathrm{d}s\right) ^{\frac{\bar{\epsilon}}{2+\bar{\epsilon}}} \\
		\leq & C_{\bar{\sigma},T}L_{\bar{\epsilon}}\left( \sup_{t\in \lbrack 0,T]}\hat{\mathbb{E}}[1+|\phi (t)|^{2}]\right) t^{\frac{\bar{\epsilon}}{2+\bar{\epsilon}}}.
	\end{align*}
	Then we have 
	\begin{equation*}
		\hat{\mathbb{E}}[|X_{n+1}(t)-X_{n}(t)|^{2}]\leq \left( \sup_{t\in \lbrack0,T]}\hat{\mathbb{E}}[1+|\phi (t)|^{2}]\right) (C_{\bar{\sigma},T}L_{\bar{\epsilon}})^{n+1}\left( \frac{t^{n+1}}{(n+1)!}\right)^{\frac{\bar{\epsilon}}{2+\bar{\epsilon}}}.
	\end{equation*}%
	For $n>m\geq 1$, by H\"{o}lder's inequality, we obtain 
	\begin{align*}
		\hat{\mathbb{E}}[|X_{n}(t)-X_{m}(t)|^{2}]=& \hat{\mathbb{E}}
		[|X_{n}(t)-X_{n-1}(t)+\cdots +X_{m+1}(t)-X_{m}(t)|^{2}] \\
		\leq&\left((\hat{\mathbb{E}}[|X_{n}(t)-X_{n-1}(t)|^{2}])^{\frac{1}{2}}+\cdots+(\hat{\mathbb{E}}[|X_{m+1}(t)-X_{m}(t)|^{2}])^{\frac{1}{2}}\right)^{2} \\
		\leq&\left(\sum_{k=m+1}^{n}(C_{\bar{\sigma},T}L_{\bar{\epsilon}})^{\frac{k}{2}}\left( \frac{t^{k}}{k!}\right) ^{\frac{\bar{\epsilon}}{4+2\bar{\epsilon}}}\right) ^{2}\sup_{t\in \lbrack 0,T]}\hat{\mathbb{E}}[1+|\phi
		(t)|^{2}].
	\end{align*}
	Since the series $\sum_{k=1}^{\infty}(C_{\bar{\sigma},T}L_{\bar{\epsilon}})^{\frac{k}{2}}\left(\frac{T^{k}}{k!}\right)^{\frac{\bar{\epsilon}}{4+2\bar{\epsilon}}}$ is convergent, we get 
	\begin{equation*}
		\sup_{t\in\lbrack0,T]}\hat{\mathbb{E}}[|X_{n}(t)-X_{m}(t)|^{2}]\leq\left(\sum_{k=m+1}^{n}(C_{\bar{\sigma},T}L_{\bar{\epsilon}})^{\frac{k}{2}}\left( \frac{T^{k}}{k!}\right) ^{\frac{\bar{\epsilon}}{4+2\bar{\epsilon}}}\right)^{2}\sup_{t\in \lbrack 0,T]}\hat{\mathbb{E}}[1+|\phi (t)|^{2}]\rightarrow
		0\quad \text{as }n,m\rightarrow \infty .
	\end{equation*}
	Then we obtain $\left( X_{n}\right) _{n\geq 0}$ is a Cauchy sequence under the norm $\left( \sup_{t\in \lbrack0,T]}\hat{\mathbb{E}}[|\cdot
	|^{2}]\right) ^{\frac{1}{2}}$. Note that $\Vert X_{n}\Vert_{M_{G}^{2}}\leq T^{\frac{1}{2}}(\sup\limits_{t\in \lbrack 0,T]}\hat{\mathbb{E}}[|X_{n}(t)|^{2}])^{\frac{1}{2}}$. Therefore, we can find $X(\cdot )$ $\in M_{G}^{2}[0,T]$ such that
	\begin{equation}
		\lim_{n\rightarrow \infty }\sup\limits_{t\in \lbrack 0,T]}\hat{\mathbb{E}}\left[ |X(t)-X_{n}(t)|^{2}\right] =0.  \label{xn-x}
	\end{equation}
	Next, we need to prove 
	\begin{equation}
		\lim_{n\rightarrow \infty }\hat{\mathbb{E}}\left[ \left\vert
		\int_{0}^{t}\sigma (t,s,X(s))-\sigma (t,s,X_{n-1}(s))\mathrm{d}%
		B(s)\right\vert ^{2}\right] =0.  \label{sigma xn}
	\end{equation}
	From (H2), it leads to
	\begin{align*}
		&\hat{\mathbb{E}}\left[ \left\vert \int_{0}^{t}\sigma (t,s,X(s))-\sigma(t,s,X_{n-1}(s))\mathrm{d}B(s)\right\vert ^{2}\right]  \\
		\leq&\bar{\sigma}^{2}\hat{\mathbb{E}}\left[\int_{0}^{t}\left\vert \sigma(t,s,X(s))-\sigma (t,s,X_{n-1}(s))\right\vert ^{2}\mathrm{d}s\right]  \\
		\leq&\bar{\sigma}^{2}\hat{\mathbb{E}}\left[\int_{0}^{t}L^{2}(t,s)\left\vert X(s)-X_{n-1}(s)\right\vert ^{2}\mathrm{d}s\right]  \\
		\leq&\bar{\sigma}^{2}\int_{0}^{t}L^{2}(t,s)\hat{\mathbb{E}}\left[
		\left\vert X(s)-X_{n-1}(s)\right\vert ^{2}\right] \mathrm{d}s \\
		\leq&\bar{\sigma}^{2}\left(\sup\limits_{t\in\lbrack0,T]}\hat{\mathbb{E}}\left[|X(t)-X_{n-1}(t)|^{2}\right]\right)\int_{0}^{t}L^{2}(t,s)\mathrm{d}s.
	\end{align*}%
	Due to (\ref{xn-x}), we obtain that (\ref{sigma xn}) holds. Similarly, we can prove that
	\begin{align*}
		&\lim_{n\rightarrow \infty }\hat{\mathbb{E}}\left[ \left\vert
		\int_{0}^{t}b(t,s,X_{n-1}(s))\mathrm{d}s+\int_{0}^{t}h(t,s,X_{n-1}(s))\mathrm{d}\langle B\rangle (s)\right\vert ^{2}\right]\\ =&\hat{\mathbb{E}}\left[\left\vert\int_{0}^{t}b(t,s,X(s))\mathrm{d}s+\int_{0}^{t}h(t,s,X(s))\mathrm{d}\left\langle B\right\rangle
		(s)\right\vert ^{2}\right] .
	\end{align*}
	Let $n\rightarrow \infty $ on both side of (\ref{Xn}), we obtain that $X$ is a solution of equation (\ref{G-SVIE}) and $\sup\limits_{t\in \lbrack 0,T]}\hat{\mathbb{E}}[|X(t)|^{2}]<\infty $. Moreover, due to Lemma \ref{sigma MG}, we obtain that $X(\cdot )-\phi (\cdot )\in \tilde{M}_{G}^{2}(0,T)$.
	
	Uniqueness. Let $X(\cdot ),\ Y(\cdot )\in M_{G}^{2}(0,T)$ be two solutions of equation (\ref{G-SVIE}) and 
	\begin{equation*}
		\sup\limits_{t\in \lbrack 0,T]}\hat{\mathbb{E}}[|X(t)|^{2}]<\infty ,\
		\sup\limits_{t\in \lbrack 0,T]}\hat{\mathbb{E}}[|Y(t)|^{2}]<\infty .
	\end{equation*}%
	By H\"{o}lder's inequality, it follows that 
	\begin{align*}
		\left(\hat{\mathbb{E}}[|X(t)-Y(t)|^{2}]\right)^{\frac{2+\bar{\epsilon}}{\bar{\epsilon}}}=& \left( \hat{\mathbb{E}}\left[ \left\vert\int_{0}^{t}\hat{b}(X(s)-Y(s))\mathrm{d}s+\int_{0}^{t}\hat{h}(X(s)-Y(s))\mathrm{d}\langle B\rangle (s)\right. \right. \right.  \\
		&+\left.\left.\left.\int_{0}^{t}\hat{\sigma}(X(s)-Y(s))\mathrm{d}B(s)\right\vert^{2}\right]\right)^{\frac{2+\bar{\epsilon}}{\bar{\epsilon}}} \\
		\leq & \left(C\left((\bar{\sigma}^{4}+1)T+\bar{\sigma}^{2}\right)
		\int_{0}^{t}L^{2}(t,s)\hat{\mathbb{E}}[|X(s)-Y(s)|^{2}]\mathrm{d}s\right)^{\frac{2+\bar{\epsilon}}{\bar{\epsilon}}} \\
		\leq&C_{\bar{\sigma},T}\left(\int_{0}^{t}L^{2+\bar{\epsilon}}(t,s)\mathrm{d}s\right)^{\frac{2}{\bar{\epsilon}}}\int_{0}^{t}\left( \hat{\mathbb{E}}[|X(s)-Y(s)|^{2}]\right)^{\frac{2+\bar{\epsilon}}{\bar{\epsilon}}}\mathrm{d}s \\
		\leq&C_{\bar{\sigma},T}L_{\bar{\epsilon}}^{\frac{2+\bar{\epsilon}}{\bar{\epsilon}}}\int_{0}^{t}\left(\hat{\mathbb{E}}[|X(s)-Y(s)|^{2}]\right) ^{\frac{2+\bar{\epsilon}}{\bar{\epsilon}}}\mathrm{d}s.
	\end{align*}
	Define $u(t)=\sup_{0\leq s\leq t}\left(\hat{\mathbb{E}}[|X(s)-Y(s)|^{2}]\right)^{\frac{2+\bar{\epsilon}}{\bar{\epsilon}}}$ for $t\in \lbrack 0,T]$. It is easy to check that
	\begin{align*}
		u(t)\leq&C_{\bar{\sigma},T}L_{\bar{\epsilon}}^{\frac{2+\bar{\epsilon}}{\bar{\epsilon}}}\int_{0}^{t}\left(\hat{\mathbb{E}}[|X(s)-Y(s)|^{2}]\right)^{\frac{2+\bar{\epsilon}}{\bar{\epsilon}}}\mathrm{d}s \\
		\leq&C_{\bar{\sigma},T}L_{\bar{\epsilon}}^{\frac{2+\bar{\epsilon}}{\bar{\epsilon}}}\int_{0}^{t}u(s)\mathrm{d}s.
	\end{align*}
	\ Using Gronwall's inequality, we have 
	\begin{equation*}
		u(t)=0, \ t\in[0,T],
	\end{equation*}
	which implies 
	\begin{equation*}
		\sup_{t\in \lbrack 0,T]}\hat{\mathbb{E}}[|X(t)-Y(t)|^{2}]=0.
	\end{equation*}
	Therefore,  
	\begin{equation*}
		\Vert X-Y\Vert _{M_{G}^{2}}\leq T^{\frac{1}{2}}(\sup\limits_{t\in \lbrack0,T]}\hat{\mathbb{E}}[|X(t)-Y(t)|^{2}])^{\frac{1}{2}}=0.
	\end{equation*}
	
	Continuity. We first proof the process $\int_{0}^{t}\sigma (t,s,X(s))\mathrm{d}B(s),\ t\in \lbrack 0,T]$, is mean-square continuous with respect to $t$. For simplicity, we may take $t>t^{\prime }$. Since $\sup\limits_{t\in
	\lbrack 0,T]}\hat{\mathbb{E}}[|X(t)|^{2}]<\infty $, according to (\ref{sigam t t'}), we have
	\begin{align*}
		&\hat{\mathbb{E}}\left[ \left\vert \int_{0}^{t}\sigma (t,s,X(s))\mathrm{d}s-\int_{0}^{t^{\prime }}\sigma (t^{\prime },s,X(s))\mathrm{d}B(s)\right\vert^{2}\right]  \\
		\leq &2\bar{\sigma}^{2}\left( \sup_{t\in \lbrack 0,T]}\hat{\mathbb{E}}\left[1+\left\vert X(t)\right\vert ^{2}\right] \right) \left\{ C_{T}\rho \left(\left\vert t-t^{\prime}\right\vert\right)+\left(\int_{t^{\prime}}^{t}L^{2+\bar{\epsilon}}(t,s)\mathrm{d}s\right) ^{\frac{2}{2+\bar{\epsilon}%
		}}\left\vert t-t^{\prime}\right\vert^{\frac{\bar{\epsilon}}{2+\bar{\epsilon
		}}}\right\}  \\
		\leq&2\bar{\sigma}^{2}\left(\sup_{t\in\lbrack0,T]}\hat{\mathbb{E}}\left[1+\left\vert X(t)\right\vert ^{2}\right] \right) \left\{ C_{T}\rho \left(\left\vert t-t^{\prime }\right\vert \right) +L_{\bar{\epsilon}}\left\vert t-t^{\prime }\right\vert ^{\frac{\bar{\epsilon}}{2+\bar{\epsilon}}}\right\} .
	\end{align*}
	It follows that
	\begin{equation*}
		\lim_{t\rightarrow t^{\prime }}{}\hat{\mathbb{E}}\left[ \left\vert\int_{0}^{t}\sigma(t,s,X(s))\mathrm{d}s-\int_{0}^{t^{\prime }}\sigma(t^{\prime },s,X(s))\mathrm{d}B(s)\right\vert ^{2}\right] =0
	\end{equation*}
	Similarly, we can get $\int_{0}^{t}b(t,s,X(s))\mathrm{d}s+\int_{0}^{t}h(t,s,X(s))\mathrm{d}\langle B\rangle (s)$ is mean-square continuous. The proof is complete.
\end{proof}

\begin{remark}
	In the above theorem, if $\phi $ still belongs to $\tilde{M}_{G}^{2}(0,T)$, then $X(\cdot )\in\tilde{M}_{G}^{2}(0,T)$. In addition, if the conditions for $\phi $ are changed to $\phi (\cdot )\in \tilde{M}_{G}^{2}(0,T)$ and is mean-square continuous, then according to Lemma 3.4 in \cite{zhao2025G-SVIE}, we know
	that $\sup_{t\in \lbrack 0,T]}\hat{\mathbb{E}}[|\phi(t)|^{2}]<\infty $. In this case, $X(\cdot )\in \tilde{M}_{G}^{2}(0,T)$ and is mean-square continuous.
\end{remark}

Next, we present the assumptions required for the path continuity of the
solution.

\begin{description}
	\item[(H4)] Let $t_{1}\geq t_{2}$. For all $x\in \mathbb{R}$ and $(t_{1},s)\in \bigtriangleup (t_{1},s),(t_{2},s)\in \bigtriangleup
	(t_{2},s)$, there exists a deterministic positive function $\bar{K}(t_{1},t_{2},s)$ such that 
	\begin{equation*}
		\left\vert \sigma (t_{1},s,x)-\sigma (t_{2},s,x)\right\vert ^{2+\epsilon}\leq \bar{K}^{2+\epsilon }(t_{1},t_{2},s)\left( 1+\left\vert x\right\vert^{2+\epsilon }\right) .
	\end{equation*}
	Here $\bar{K}(t_{1},t_{2},s)$ satisfies $\int_{0}^{t_{2}}\bar{K}^{2+\epsilon}(t_{1},t_{2},s)\mathrm{d}s\leq C_{T}\left\vert t_{1}-t_{2}\right\vert^{\alpha }$ where $\alpha >1$.
\end{description}

\begin{theorem}
	Let $0<\epsilon <\bar{\epsilon}$. Assume that $\phi (\cdot )\in
	M_{G}^{2+\epsilon }(0,T)$ satisfies 
	\begin{equation*}
		\sup_{t\in \lbrack 0,T]}\hat{\mathbb{E}}[|\phi (t)|^{2+\epsilon }]<\infty .
	\end{equation*}
	If (H1)-(H3) hold, then there exists a unique solution $X(\cdot )\in
	M_{G}^{2+\epsilon }(0,T)$ of (\ref{G-SVIE}) satisfies $\sup\limits_{t\in\lbrack 0,T]}\hat{\mathbb{E}}[|X(t)|^{2+\epsilon }]<\infty $ and $X(\cdot)-\phi (\cdot )\in \tilde{M}_{G}^{2+\epsilon }(0,T)$. Moreover, if (H4)
	holds and $\bar{\epsilon}\epsilon >4$, then $X(\cdot )-\phi (\cdot )$ has a continuous modification.
\end{theorem}

\begin{proof}[Proof]
	First, we need to prove that there exists $X(\cdot )\in M_{G}^{2+\epsilon}(0,T)$ such that $\sup\limits_{t\in \lbrack 0,T]}\hat{\mathbb{E}}
	[|X(t)|^{2+\epsilon }]<\infty $ and $X(\cdot )-\phi (\cdot )\in \tilde{M}_{G}^{2+\epsilon }(0,T)$. This statement can be proved similarly to the method employed in the proof of Theorem \ref{solution}.
	
	Now we prove the path continuity of $\int_{0}^{t}\sigma(t,s,X(s))\mathrm{d}B(s),\ t\in \lbrack 0,T]$. Without loss of generality, we consider $t>\bar{t}$. By Theorem \ref{BDG} and H\"{o}lder's inequality, we obtain 
	\begin{align*}
		& \hat{\mathbb{E}}\left[ \left\vert \int_{0}^{t}\sigma (t,s,X(s))\mathrm{d}B(s)-\int_{0}^{\bar{t}}\sigma (\bar{t},s,X(s))\mathrm{d}B(s)\right\vert
		^{2+\epsilon }\right] \\
		=& \hat{\mathbb{E}}\left[ \left\vert \int_{0}^{\bar{t}}\sigma
		(t,s,X(s))-\sigma(\bar{t},s,X(s))\mathrm{d}B(s)+\int_{\bar{t}}^{t}\sigma(t,s,X(s))\mathrm{d}B(s)\right\vert ^{2+\epsilon }\right] \\
		\leq & C_{\epsilon }\left( \hat{\mathbb{E}}\left[ \left\vert \int_{0}^{\bar{t}}\sigma(t,s,X(s))-\sigma(\bar{t},s,X(s))\mathrm{d}B(s)\right\vert^{2+\epsilon }\right] +\hat{\mathbb{E}}\left[ \left\vert \int_{\bar{t}}^{t}\sigma (t,s,X(s))dB(s)\right\vert ^{2+\epsilon }\right] \right) \\
		\leq & C_{\bar{\sigma},\epsilon }\left\{ \hat{\mathbb{E}}\left[ \left\vert\int_{0}^{\bar{t}}\left\vert \sigma (t,s,X(s))-\sigma (\bar{t},s,X(s))\right\vert ^{2}\mathrm{d}s\right\vert ^{\frac{2+\epsilon }{2}}\right] +\hat{\mathbb{E}}\left[ \left\vert \int_{\bar{t}}^{t}\left\vert\sigma (t,s,X(s))\right\vert ^{2}\mathrm{d}s\right\vert ^{\frac{2+\epsilon }{2}}\right] \right\} \\
		\leq & C_{\bar{\sigma},T,\epsilon }\left\{ \hat{\mathbb{E}}\left[ \int_{0}^{\bar{t}}\left\vert \sigma (t,s,X(s))-\sigma (\bar{t},s,X(s))\right\vert^{2+\epsilon }\mathrm{d}s\right] +\hat{\mathbb{E}}\left[ \int_{\bar{t}
		}^{t}\left\vert \sigma (t,s,X(s))\right\vert ^{2+\epsilon }\mathrm{d}s\right]\left\vert t-\bar{t}\right\vert ^{\frac{\epsilon }{2}}\right\} \\
		\leq & C_{\bar{\sigma},T,\epsilon }\left\{ \hat{\mathbb{E}}\left[ \int_{0}^{\bar{t}}\bar{K}^{2+\epsilon }(t,\bar{t},s)\left( 1+\left\vert X(s)\right\vert ^{2+\epsilon }\right) \mathrm{d}s\right] +\hat{\mathbb{E}}
		\left[ \int_{\bar{t}}^{t}L^{2+\epsilon }(t,s)\left( 1+\left\vert
		X(s)\right\vert ^{2+\epsilon }\right) \mathrm{d}s\right] \left\vert t-\bar{t}\right\vert ^{\frac{\epsilon }{2}}\right\} \\
		\leq &C_{\bar{\sigma},T,\epsilon}\left\{\int_{0}^{\bar{t}}\bar{K}
		^{2+\epsilon }(t,\bar{t},s)\hat{\mathbb{E}}\left[ 1+\left\vert X(s)\right\vert ^{2+\epsilon }\right] \mathrm{d}s+\left\vert t-\bar{t}\right\vert^{\frac{\epsilon}{2}}\int_{\bar{t}}^{t}L^{2+\epsilon }(t,s)\hat{\mathbb{E}}\left[ 1+\left\vert X(s)\right\vert ^{2+\epsilon }\right] \mathrm{d}s\right\} \\
		\leq & C_{\bar{\sigma},T,\epsilon }\left\{ \left( \sup_{t\in \lbrack 0,T]}\hat{\mathbb{E}}\left[ 1+\left\vert X(t)\right\vert ^{2+\epsilon}\right]\right)\int_{0}^{\bar{t}}\bar{K}^{2+\epsilon }(t,t_{i}^{m},s)\mathrm{d}s\right. \\
		&+\left. \left\vert t-\bar{t}\right\vert ^{\frac{\epsilon }{2}}\left( \int_{\bar{t}}^{t}L^{2+\bar{\epsilon}}(t,s)\mathrm{d}s\right) ^{\frac{2+\epsilon }{2+\bar{\epsilon}}}\left( \int_{\bar{t}}^{t}\left( \hat{\mathbb{E}}\left[
		1+\left\vert X(s)\right\vert ^{2+\epsilon }\right] \right) ^{\frac{2+\bar{\epsilon}}{\bar{\epsilon}-\epsilon }}\mathrm{d}s\right) ^{\frac{\bar{\epsilon}-\epsilon }{2+\bar{\epsilon}}}\right\} \\
		\leq & C_{\bar{\sigma},T,\epsilon }\left( \sup_{t\in \lbrack 0,T]}\hat{\mathbb{E}}\left[ 1+\left\vert X(t)\right\vert ^{2+\epsilon }\right] \right)\left\{ C_{T}\left\vert t-\bar{t}\right\vert ^{\alpha }+\left( \sup_{t\in\lbrack 0,T]}\int_{0}^{T}L^{2+\bar{\epsilon}}(t,s)\mathrm{d}s\right) ^{\frac{2+\epsilon }{2+\bar{\epsilon}}}\left\vert t-\bar{t}\right\vert ^{\frac{\bar{\epsilon}\left( 2+\epsilon \right) }{2\left( 2+\bar{\epsilon}\right) }
		}\right\} .
	\end{align*}
	Note that both $\alpha $ and $\frac{\bar{\epsilon}\left( 2+\epsilon \right) }{2\left( 2+\bar{\epsilon}\right) }$ are greater than $1$. Then using Theorem \ref{continuous}, we yield that there admits a $\delta $-order H\"{o}lder's continuous modification of $\int_{0}^{t}\sigma (t,s,X(s))\mathrm{d}B(s),\ t\in \lbrack 0,T]$ for $\delta \in \left[ 0,\frac{\alpha -1}{2+\epsilon }\wedge \frac{\bar{\epsilon}\epsilon -4}{2\left( 2+\bar{\epsilon}\right) \left( 2+\epsilon \right) }\right) $.
	
	Finally, we deal with 
	\begin{equation*}
		t\mapsto \int_{0}^{t}b(t,s,X(s))\mathrm{d}s+\int_{0}^{t}h(t,s,X(s))\mathrm{d}%
		\langle B\rangle (s).
	\end{equation*}
	Since $\left( a+b\right)^{\frac{1}{2+\epsilon }}\leq a^{\frac{1}{2+\epsilon }}+b^{\frac{1}{2+\epsilon }}$, for $t>\bar{t}$ and each $\omega \in \Omega $,\ we have
	\begin{align*}
		&\left\vert\int_{0}^{t}b(t,s,X(s))ds-\int_{0}^{\bar{t}}b(\bar{t},s,X(s))\mathrm{d}s\right\vert \\
		\leq&\left\vert\int_{0}^{\bar{t}}b(t,s,X(s))-b(\bar{t},s,X(s))\mathrm{d}s\right\vert+\left\vert\int_{\bar{t}}^{t}b(t,s,X(s))\mathrm{d}s\right\vert\\
		\leq &\left\vert \int_{0}^{\bar{t}}K(t,\bar{t},s)\left( 1+\left\vert X(s)\right\vert \right) \mathrm{d}s\right\vert +\left\vert \int_{\bar{t}}^{t}L(t,s)\left( 1+\left\vert X(s)\right\vert \right) \mathrm{d}s\right\vert\\
		\leq&\left(\int_{0}^{\bar{t}}K^{2}(t,\bar{t},s)\mathrm{d}s\right)^{\frac{1}{2}}\left( \int_{0}^{\bar{t}}\left( 1+\left\vert X(s)\right\vert \right)^{2}\mathrm{d}s\right) ^{\frac{1}{2}}+\left(\int_{\bar{t}}^{t}L^{2}(t,s)\mathrm{d}s\right) ^{\frac{1}{2}}\left( \int_{\bar{t}}^{t}\left(1+\left\vert X(s)\right\vert \right) ^{2}\mathrm{d}s\right) ^{\frac{1}{2}} \\
		\leq&C_{T,\epsilon}\left\{\left(\int_{0}^{\bar{t}}\left(1+\left\vert X(s)\right\vert\right)^{2}\mathrm{d}s\right)^{\frac{1}{2}}\left(\int_{0}^{\bar{t}}K^{2+\epsilon}(t,\bar{t},s)\mathrm{d}s\right)^{\frac{1}{2+\epsilon}}\right. \\
		&\left. +\left(\int_{\bar{t}}^{t}L^{2}(t,s)\mathrm{d}s\right) ^{\frac{1}{2}}\left( \int_{\bar{t}}^{t}\left( 1+\left\vert X(s)\right\vert \right)^{2+\epsilon }\mathrm{d}s\right) ^{\frac{1}{2+\epsilon }}\left\vert t-\bar{t}\right\vert ^{\frac{\epsilon }{2\left( 2+\epsilon \right) }}\right\} \\
		\leq &C_{T,\epsilon }\left\{ \left( \int_{0}^{T}\left( 1+\left\vert X(s)\right\vert \right) ^{2}\mathrm{d}s\right) ^{\frac{1}{2}}\rho ^{\frac{1}{2+\epsilon }}\left( \left\vert t-\bar{t}\right\vert \right) \right. \\
		&\left. +\left(\sup_{t\in \lbrack 0,T]}\int_{0}^{T}L^{2}(t,s)\mathrm{d}s\right) ^{\frac{1}{2%
		}}\left( \int_{0}^{T}\left( 1+\left\vert X(s)\right\vert \right)
		^{2+\epsilon }\mathrm{d}s\right) ^{\frac{1}{2+\epsilon }}\left\vert t-\bar{t}\right\vert ^{\frac{\epsilon }{2\left( 2+\epsilon \right) }}\right\} .
	\end{align*}
	Due to $X(\cdot )\in M_{G}^{2+\epsilon }(0,T)$, we get that $\left(\int_{0}^{T}\left( 1+\left\vert X(s)\right\vert \right) ^{2}ds\right) ^{\frac{1}{2}}+\left( \int_{0}^{T}\left( 1+\left\vert X(s)\right\vert \right)^{2+\epsilon }ds\right) ^{\frac{1}{2+\epsilon }}<\infty $ q.s. Thus,
	\begin{equation*}
		\lim_{t\rightarrow\bar{t}}\left\vert\int_{0}^{t}b(t,s,X(s))ds-\int_{0}^{\bar{t}}b(\bar{t},s,X(s))\mathrm{d}s\right\vert =0.
	\end{equation*}
	Similarly, we can obtain 
	\begin{equation*}
		\lim_{t\rightarrow\bar{t}}\left\vert\int_{0}^{t}h(t,s,X(s))\mathrm{d}\langle B\rangle(s)-\int_{0}^{\bar{t}}h(\bar{t},s,X(s))\mathrm{d}\langle B\rangle (s)\right\vert =0.
	\end{equation*}%
	It follows that $X(\cdot )-\phi (\cdot )$ has continuous paths q.s. The	proof is completed.
\end{proof}

\begin{remark}
	In the above theorem, if $\phi (\cdot )$ still has continuous paths, then $X(\cdot )$ has a continuous modification.
\end{remark}

\section{$G$-SVIE with non-Lipschitz coefficients}\label{sec4 G-SVIE non-lip}

In this section, we study the existence and uniqueness of the solution to
equation (\ref{G-SVIE}) under more general conditions. Of course, our
assumptions also need to change. For convenience, here we consider the
simple case.

\begin{description}
	\item[(H1')] For each $t\in \lbrack 0,T]$ and $x\in 
	\mathbb{R},\ b(t,\cdot ,x),\ h(t,\cdot ,x),\ \sigma (t,\cdot ,x)\in
	M_{G}^{2}(0,t)$.
	
	\item[(H2')] For all $x,y\in \mathbb{R}$ and $(t,s)\in \bigtriangleup (t,s)$, 
	\begin{align*}
		&\left\vert b(t,s,x)-b(t,s,y)\right\vert ^{2}+\left\vert h(t,s,x)-h(t,s,y)\right\vert^{2}+\left\vert\sigma(t,s,x)-\sigma(t,s,y)\right\vert ^{2} \leq \psi \left( \left\vert x-y\right\vert^{2}\right) , \\
		&\left\vert b(t,s,x)\right\vert ^{2}+\left\vert h(t,s,x)\right\vert^{2}+\left\vert \sigma (t,s,x)\right\vert ^{2} \leq L^{2}\left(1+\left\vert x\right\vert ^{2}\right) .
	\end{align*}
	where $L\in \mathbb{R}^{+}$ is constant and $\psi :\left[ 0,\infty \right) \rightarrow \left[0,\infty \right) $ is a continuous increasing and concave function satisfies
	\begin{equation*}
		\int_{0^{+}}\frac{\mathrm{d}s}{\psi \left( s\right) }=+\infty , \quad\psi \left( 0\right) =0.
	\end{equation*}
	
	\item[(H3')] For all $x,\alpha \in \mathbb{R}$ and $(t_{1},s)\in \bigtriangleup (t_{1},s),(t_{2},s)\in \bigtriangleup (t_{2},s)$, 
	\begin{equation*}
		\left\vert b(t_{1},s,x)-b(t_{2},s,x)\right\vert ^{2}+\left\vert
		h(t_{1},s,x)-h(t_{2},s,x)\right\vert^{2}+\left\vert\sigma(t_{1},s,x)-\sigma (t_{2},s,x)\right\vert ^{2}\leq \rho \left( \left\vert
		t_{1}-t_{2}\right\vert \right) ,
	\end{equation*}
	where $\rho $\ is defined as in (H3).
\end{description}

Since the coefficients no longer satisfy (H2), we still need to verify that the $G$-stochastic integral is well-defined.

\begin{lemma}\label{nonlip sigmaMG} 
	Define $M(t)=\int_{0}^{t}\sigma (t,s,X(s))\mathrm{d}B(s)$ and $N(t)=\int_{0}^{t}b(t,s,X(s))\mathrm{d}s+\int_{0}^{t}h(t,s,X(s))
	\mathrm{d}\langle B\rangle (s)$. Assume that $X(\cdot )\in M_{G}^{2}(0,T)$. If (H1')-(H3') hold, then $M(\cdot )+N(\cdot )\in \tilde{M}_{G}^{2}(0,T)$.
\end{lemma}

\begin{proof}
	According to Lemma 3.4 in \cite{bai2014on}, we konw that $\sigma (t,\cdot ,X(\cdot))\in M_{G}^{2}(0,t)$ for each $t$. It follows that  $M(t)\in L_{G}^{2}(\Omega _{t})$. Next, similar to the proof of Lemma 3.2 in \cite{zhao2025G-SVIE}, we can obtain $M(\cdot )\in M_{G}^{2}(0,T)$. Therefore, $M(\cdot)\in \tilde{M}_{G}^{2}(0,T)$. Similarly, we can prove that $N(\cdot )\in \tilde{M}_{G}^{2}(0,T)$.
\end{proof}

In the following, we present the main result of this section.

\begin{theorem}
	Assume that $\phi (\cdot )\in \tilde{M}_{G}^{2}(0,T)$ satisfies 
	\begin{equation*}
		\sup_{t\in \lbrack 0,T]}\hat{\mathbb{E}}[|\phi (t)|^{2}]<\infty .
	\end{equation*}
	If (H1' )-(H3') hold, then there exists a unique solution $X(\cdot )\in \tilde{M}_{G}^{2}(0,T)$ of (\ref{G-SVIE}) satisfies $\sup\limits_{t\in \lbrack 0,T]}\hat{\mathbb{E}}[|X(t)|^{2}]<\infty $. Moreover, $X(\cdot )$ is mean-square continuous.
\end{theorem}

\begin{proof}
	Existence. Here, we still prove the existence of the solution through the Picard iteration method. The iterative equation is the same as (\ref{Xn}). From Lemma \ref{nonlip sigmaMG} and $X_{0}(\cdot )=\phi (\cdot )\in \tilde{M}_{G}^{2}(0,T)$, we have $X_{n}(\cdot )$ $\in \tilde{M}_{G}^{2}(0,T)$. Through (H1')-(H2') and similar analyses as in (\ref{supXn}), 
	\begin{align*}
		\hat{\mathbb{E}}[|X_{n}(t)|^{2}] =&\hat{\mathbb{E}}\left[ \left\vert\phi(t)+\int_{0}^{t}b(t,s,X_{n-1}(s))\mathrm{d}s+\int_{0}^{t}h(t,s,X_{n-1}(s))\mathrm{d}\langle B\rangle (s)+\int_{0}^{t}\sigma (t,s,X_{n-1}(s))\mathrm{d}B(s)\right\vert ^{2}\right]  \\
		\leq &{}C\left( \hat{\mathbb{E}}[|\phi (t)|^{2}]+\left( (\bar{\sigma}^{4}+1)T+\bar{\sigma}^{2}\right) \hat{\mathbb{E}}\left[\int_{0}^{t}L^{2}(1+|X_{n-1}(s)|^{2})\mathrm{d}s\right] \right)  \\
		\leq &{}C_{\bar{\sigma},T,L}\left( \sup_{t\in \lbrack 0,T]}\hat{\mathbb{E}}[|\phi (t)|^{2}]+\hat{\mathbb{E}}\left[ \int_{0}^{t}|X_{n-1}(s)|^{2}\mathrm{d}s\right] \right)  \\
		\leq &{}C_{\bar{\sigma},T,L}\left( \sup_{t\in \lbrack 0,T]}\hat{\mathbb{E}}[|\phi(t)|^{2}]+\int_{0}^{t}\hat{\mathbb{E}}\left[ |X_{n-1}(s)|^{2}\right] \mathrm{d}s\right) 
	\end{align*}
	Set $M_{\phi }=\sup_{t\in \lbrack 0,T]}\hat{\mathbb{E}}[|\phi (t)|^{2}]$. Then 
	\begin{align*}
		\hat{\mathbb{E}}[|X_{1}(t)|^{2}] \leq &C_{\bar{\sigma},T,L}\left(
		\sup_{t\in\lbrack0,T]}\hat{\mathbb{E}}[|\phi(t)|^{2}]+\hat{\mathbb{E}}\left[ \int_{0}^{t}|X_{0}(s)|^{2}\mathrm{d}s\right] \right)  \\
		\leq &C_{\bar{\sigma},T,L}M_{\phi }\left( t+1\right) .
	\end{align*}
	By standard calculation, we yield 
	\begin{align*}
		\hat{\mathbb{E}}[|X_{n}(t)|^{2}] \leq&C_{\bar{\sigma},T,L}M_{\phi }\exp(C_{\bar{\sigma},T,L}t)+M_{\phi }\frac{\left( C_{\bar{\sigma},T,L}t\right)^{n}}{n!} \\
		\leq&C_{\bar{\sigma},T,L}M_{\phi}\exp(C_{\bar{\sigma},T,L}T)+M_{\phi }\frac{\left( C_{\bar{\sigma},T,L}T\right) ^{n}}{n!},
	\end{align*}
	which implies 
	\begin{equation}
		\sup_{n}\sup_{t\in\lbrack0,T]}\hat{\mathbb{E}}[|X_{n}(t)|^{2}]<\infty. \label{supn supt xn}
	\end{equation}%
	Denote 
	\begin{equation*}
		\hat{b}(x-y):=b(t,s,x)-b(t,s,y),\ \hat{h}(x-y):=h(t,s,x)-h(t,s,y),\ \hat{\sigma}(x-y):=\sigma (t,s,x)-\sigma (t,s,y).
	\end{equation*}
	Due to (H2'), it follows that 
	\begin{align*}
		&\hat{\mathbb{E}}[|X_{n+m+1}(t)-X_{m+1}(t)|^{2}] \\
		\leq &\hat{\mathbb{E}}\left[\left\vert \int_{0}^{t}\hat{b}(X_{n+m}(s)-X_{m}(s))\mathrm{d}s+\int_{0}^{t}\hat{h}(X_{n+m}(s)-X_{m}(s))\mathrm{d}\langle B\rangle (s)+ \int_{0}^{t}\hat{\sigma}(X_{n+m}(s)-X_{m}(s))\mathrm{d}B(s)\right\vert ^{2}\right]  \\
		\leq &C\left( \hat{\mathbb{E}}\left[ \left\vert \int_{0}^{t}\hat{b}(X_{n+m}(s)-X_{m}(s))\mathrm{d}s\right\vert ^{2}\right] +\hat{\mathbb{E}}\left[ \left\vert \int_{0}^{t}\hat{h}(X_{n+m}(s)-X_{m}(s))\mathrm{d}\langle B\rangle (s)\right\vert ^{2}\right] \right.  \\
		&+\left. \hat{\mathbb{E}}\left[ \left\vert\int_{0}^{t}\hat{\sigma}
		(X_{n+m}(s)-X_{m}(s))\mathrm{d}B(s)\right\vert ^{2}\right] \right)  \\
		\leq &C_{\bar{\sigma},T}\hat{\mathbb{E}}\left[ \int_{0}^{t}\psi \left(\left\vert X_{n+m}(s)-X_{m}(s)\right\vert ^{2}\right) \mathrm{d}s\right]  \\
		\leq &C_{\bar{\sigma},T}\int_{0}^{t}\hat{\mathbb{E}}\left[ \psi \left(\left\vert X_{n+m}(s)-X_{m}(s)\right\vert ^{2}\right) \right] \mathrm{d}s.
	\end{align*}
	Applying Lemma \ref{Jensen's inequality}, we have
	\begin{align*}
		\hat{\mathbb{E}}[|X_{n+m+1}(t)-X_{m+1}(t)|^{2}] \leq &C_{\bar{\sigma},T}\int_{0}^{t}\psi \left( \hat{\mathbb{E}}\left[ \left\vert X_{n+m}(s)-X_{m}(s)\right\vert ^{2}\right] \right) \mathrm{d}s \\
		\leq&C_{\bar{\sigma},T}\psi\left(\int_{0}^{t}\hat{\mathbb{E}}\left[\left\vert X_{n+m}(s)-X_{m}(s)\right\vert ^{2}\right] \mathrm{d}s\right).
	\end{align*}
	Integrate simultaneously on both sides, we get
	\begin{equation}
		\int_{0}^{t}\hat{\mathbb{E}}[|X_{n+m+1}(s)-X_{m+1}(s)|^{2}]\mathrm{d}s\leq C_{\bar{\sigma},T}\int_{0}^{t}\psi \left( \int_{0}^{s}\hat{\mathbb{E}}\left[\left\vert X_{n+m}(r)-X_{m}(r)\right\vert ^{2}\right] \mathrm{d}r\right) 
		\mathrm{d}s.  \label{nonlip Xnm-Xm}
	\end{equation}%
	Put $u_{n+m+1,m+1}(t)=\int_{0}^{t}\hat{\mathbb{E}}[|X_{n+m+1}(s)-X_{m+1}(s)|^{2}]\mathrm{d}s$. Then 
	\begin{equation*}
		u_{n+m+1,m+1}(t)\leq C_{\bar{\sigma},T}\int_{0}^{t}\psi\left(u_{n+m,m}(s)\right) \mathrm{d}s.
	\end{equation*}%
	From (\ref{supn supt xn}), it follows that 
	\begin{equation*}
		\sup_{n,m}\sup_{t\in \lbrack 0,T]}u_{n+m+1,m+1}(t)<\infty .
	\end{equation*}
	Define $u(t):=\limsup_{n,m\rightarrow \infty }u_{n,m}(t)$. Due to
	Fatou's lemma, we deduce that 
	\begin{align*}
		u(t) \leq &C_{\bar{\sigma},T}\limsup_{n,m\rightarrow \infty}\int_{0}^{t}\psi \left( u_{n+m,m}(s)\right) \mathrm{d}s \\
		\leq & C_{\bar{\sigma},T}\int_{0}^{t}\limsup_{n,m\rightarrow \infty}\psi \left( u_{n+m,m}(s)\right) \mathrm{d}s \\
		\leq & C_{\bar{\sigma},T}\int_{0}^{t}\psi \left( u(s)\right) \mathrm{d}s.
	\end{align*}
	By Lemma \ref{Bihari's inequality}, it follows that $u(t)=0$ for $t\in \lbrack 0,T]$. Thus, we obtain $\left( X_{n}\right) _{n\geq 0}$ is a Cauchy sequence under the norm $\left(\int_{0}^{T}\hat{\mathbb{E}}[|X_{n}(t)|^{2}]\mathrm{d}t\right)^{\frac{1}{2}}$. Note that $\Vert X_{n}\Vert _{M_{G}^{2}}\leq \left(
	\int_{0}^{T}\hat{\mathbb{E}}[|X_{n}(t)|^{2}]\mathrm{d}t\right)^{\frac{1}{2}}$. Therefore, we can find $X(\cdot )$ $\in M_{G}^{2}[0,T]$ such that
	\begin{equation}
		\lim_{n\rightarrow \infty }\int_{0}^{T}\hat{\mathbb{E}}\left[
		|X(t)-X_{n}(t)|^{2}\right] \mathrm{d}t=0.  \label{nonlip X-Xn}
	\end{equation}%
	Next, we need to prove 
	\begin{equation}
		\lim_{n\rightarrow\infty}\hat{\mathbb{E}}\left[\left\vert\int_{0}^{t}\sigma(t,s,X(s))-\sigma(t,s,X_{n-1}(s))\mathrm{d}B(s)\right\vert ^{2}\right] =0.  \label{nonlip sigma xn}
	\end{equation}%
	From (H2') and (\ref{nonlip X-Xn}), it leads to
	\begin{align*}
		&\hat{\mathbb{E}}\left[ \left\vert \int_{0}^{t}\sigma (t,s,X(s))-\sigma(t,s,X_{n-1}(s))\mathrm{d}B(s)\right\vert ^{2}\right]  \\
		\leq&\bar{\sigma}^{2}\hat{\mathbb{E}}\left[\int_{0}^{t}\left\vert \sigma(t,s,X(s))-\sigma (t,s,X_{n-1}(s))\right\vert ^{2}\mathrm{d}s\right]  \\
		\leq &\bar{\sigma}^{2}\hat{\mathbb{E}}\left[ \int_{0}^{t}\psi \left(\left\vert X(s)-X_{n-1}(s)\right\vert ^{2}\right) \mathrm{d}s\right]  \\
		\leq &\bar{\sigma}^{2}\int_{0}^{t}\hat{\mathbb{E}}\left[ \psi \left(\left\vert X(s)-X_{n-1}(s)\right\vert ^{2}\right) \right] \mathrm{d}s \\
		\leq &\bar{\sigma}^{2}\int_{0}^{t}\psi\left(\hat{\mathbb{E}}\left[
		\left\vert X(s)-X_{n-1}(s)\right\vert ^{2}\right] \right) \mathrm{d}s \\
		\leq &\bar{\sigma}^{2}\psi\left(\int_{0}^{t}\hat{\mathbb{E}}\left[
		\left\vert X(s)-X_{n-1}(s)\right\vert^{2}\right]\mathrm{d}s\right) .
	\end{align*}
	According to (\ref{nonlip X-Xn}), we obtain that (\ref{nonlip sigma xn}) holds. Similarly, we can prove that
	\begin{align*}
		&\lim_{n\rightarrow \infty }\hat{\mathbb{E}}\left[ \left\vert
		\int_{0}^{t}b(t,s,X_{n-1}(s))\mathrm{d}s+\int_{0}^{t}h(t,s,X_{n-1}(s))\mathrm{d}\langle B\rangle (s)\right\vert ^{2}\right] \\
		=&\hat{\mathbb{E}}\left[\left\vert\int_{0}^{t}b(t,s,X(s))\mathrm{d}s+\int_{0}^{t}h(t,s,X(s))\mathrm{d}\left\langle B\right\rangle
		(s)\right\vert ^{2}\right] .
	\end{align*}
	Let $n\rightarrow \infty $ on both side of (\ref{Xn}), we obtain that $X$ is a solution of equation (\ref{G-SVIE}) and $\sup\limits_{t\in \lbrack 0,T]}\hat{\mathbb{E}}[|X(t)|^{2}]<\infty $. Moreover, $X(\cdot )\in \tilde{M}_{G}^{2}(0,T)$.
	
	Uniqueness. Let $X(\cdot ),\ Y(\cdot )\in \tilde{M}_{G}^{2}(0,T)$ be two solutions of equation (\ref{G-SVIE}). Similar to (\ref{nonlip Xnm-Xm}), we can obtain 
	\begin{equation}
		\int_{0}^{T}\hat{\mathbb{E}}[|X(t)-Y(t)|^{2}]\mathrm{d}t\leq C_{\bar{\sigma},T}\int_{0}^{T}\psi\left(\int_{0}^{t}\hat{\mathbb{E}}\left[ \left\vert X(s)-Y(s)\right\vert ^{2}\right] \mathrm{d}s\right) \mathrm{d}t.  \notag
	\end{equation}
	Due to Lemma \ref{Bihari's inequality}, we obtain  
	\begin{equation*}
		\int_{0}^{T}\hat{\mathbb{E}}[|X(t)-Y(t)|^{2}]\mathrm{d}t=0.
	\end{equation*}%
	Thus, 
	\begin{equation*}
		\Vert X-Y\Vert _{M_{G}^{2}}\leq \left(\int_{0}^{T}\hat{\mathbb{E}}
		[|X(t)-Y(t)|^{2}]\mathrm{d}t\right) ^{\frac{1}{2}}=0.
	\end{equation*}
	
	Continuity. Similar to the proof of Theorem \ref{solution}, we obtain that $X(\cdot)-\phi(\cdot)\in \tilde{M}_{G}^{2}(0,T)$ is mean-square continuous. The proof is complete.
\end{proof}

\section{$G$-SVIE with a parameter}\label{sec5 G-SVIE para}

In this section, we consider the following equation: for $0\leq T<\infty $
and $\alpha \in \mathbb{R}$,
\begin{equation}
	X_{\alpha }(t)=\phi_{\alpha}(t)+\int_{0}^{t}b_{\alpha}(t,s,X_{\alpha }(s))\mathrm{d}s+\int_{0}^{t}h_{\alpha }(t,s,X_{\alpha }(s))\mathrm{d}\langle B\rangle (s)+\int_{0}^{t}\sigma _{\alpha }(t,s,X_{\alpha}(s))\mathrm{d}B(s),\ t\in \lbrack 0,T].  \label{G-SVIE alpha}
\end{equation}
Here $b_{\alpha }(\omega ,t,s,x),h_{\alpha }(\omega ,t,s,x),\sigma _{\alpha}(\omega ,t,s,x):\mathbb{R}\times \Omega \times \bigtriangleup (t,s)\times \mathbb{R}\rightarrow \mathbb{R}$ and $\phi _{\alpha }(\omega ,t):\mathbb{R}\times \Omega \times \lbrack 0,T]\rightarrow \mathbb{R}$. For convenience, here we consider a simple case.

\begin{description}
	\item[(H1'')] For each $t\in \lbrack 0,T]$ and $x,\alpha \in \mathbb{R},\ b_{\alpha }(t,\cdot ,x),\ h_{\alpha }(t,\cdot ,x),\ \sigma _{\alpha}(t,\cdot ,x)\in M_{G}^{2}(0,t)$.
	
	\item[(H2'')] For all $x,y,\alpha \in \mathbb{R}$ and $(t,s)\in \bigtriangleup (t,s)$, there exists a positive constant $L$
	(independently on $\alpha $) such that 
	\begin{align*}
		&\left\vert b_{\alpha }(t,s,x)-b_{\alpha }(t,s,y)\right\vert +\left\vert h_{\alpha }(t,s,x)-h_{\alpha }(t,s,y)\right\vert +\left\vert \sigma _{\alpha}(t,s,x)-\sigma _{\alpha }(t,s,y)\right\vert \leq L\left\vert
		x-y\right\vert , \\
		&\left\vert b_{\alpha }(t,s,x)\right\vert ^{2}+\left\vert h_{\alpha}(t,s,x)\right\vert ^{2}+\left\vert \sigma _{\alpha }(t,s,x)\right\vert ^{2}\leq L\left( 1+\left\vert x\right\vert ^{2}\right) .
	\end{align*}
	
	\item[(H3'')] For all $x,\alpha \in \mathbb{R}$ and $(t_{1},s)\in \bigtriangleup (t_{1},s),(t_{2},s)\in \bigtriangleup(t_{2},s)$, 
	\begin{equation*}
		\left\vert b_{\alpha }(t_{1},s,x)-b_{\alpha }(t_{2},s,x)\right\vert^{2}+\left\vert h_{\alpha }(t_{1},s,x)-h_{\alpha }(t_{2},s,x)\right\vert^{2}+\left\vert \sigma_{\alpha}(t_{1},s,x)-\sigma_{\alpha}(t_{2},s,x)\right\vert ^{2}\leq \rho \left( \left\vert t_{1}-t_{2}\right\vert \right) ,
	\end{equation*}
	where $\rho $\ is defined as in (H3).
	
	\item[(H5)] For all $x,\alpha ,\beta \in \mathbb{R}$ and $(t,s)\in \bigtriangleup (t,s)$, there exists a constant $\bar{L}$ such that
	\begin{equation*}
		\left\vert \phi _{\alpha }(t)-\phi _{\beta }(t)\right\vert +\left\vert b_{\alpha }(t,s,x)-b_{\beta }(t,s,x)\right\vert +\left\vert h_{\alpha}(t,s,x)-h_{\beta }(t,s,x)\right\vert +\left\vert \sigma _{\alpha}(t,s,x)-\sigma _{\beta }(t,s,x)\right\vert \leq \bar{L}\left\vert \alpha
		-\beta \right\vert .
	\end{equation*}
\end{description}

\begin{theorem}
	Assume that $\phi _{\alpha }(\cdot )\in \tilde{M}_{G}^{2}(0,T)$ satisfies 
	\begin{equation*}
		\sup_{t\in \lbrack 0,T]}\hat{\mathbb{E}}[|\phi (t)|^{2}]<\infty .
	\end{equation*}
	If (H1'')-(H3'') and (H5) hold, then there exists a unique solution $X_{\alpha}(\cdot )\in \tilde{M}_{G}^{2}(0,T)$ of (\ref{G-SVIE alpha}) satisfies $\sup\limits_{t\in\lbrack0,T]}\hat{\mathbb{E}}[|X_{\alpha }(t)|^{2}]<\infty $. Moreover, $X_{\alpha }(t)$ is mean-square continuous with respect to $t$ and has a quasi-continuous modification with respect to $\alpha $.
\end{theorem}

\begin{proof}
	According to Theorem \ref{solution}, we know there exists a unique solution $ X_{\alpha }(\cdot )\in \tilde{M}_{G}^{2}(0,T)$ of (\ref{G-SVIE alpha}) for each $\alpha $. In addition, $X_{\alpha }(\cdot )$ is mean-square continuous. By (H2'') and (H5), it follows that
	\begin{align*}
		\left\vert b_{\alpha }(t,s,x)-b_{\beta }(t,s,y)\right\vert  \leq
		&\left\vert b_{\alpha }(t,s,x)-b_{\beta }(t,s,x)\right\vert +\left\vert b_{\beta }(t,s,x)-b_{\beta }(t,s,y)\right\vert  \\
		\leq &\bar{L}\left\vert \alpha -\beta \right\vert +L\left\vert
		x-y\right\vert .
	\end{align*}
	Similarly, we can obtain 
	\begin{align*}
		&\left\vert h_{\alpha }(t,s,x)-h_{\beta }(t,s,y)\right\vert +\left\vert\sigma _{\alpha }(t,s,x)-\sigma _{\beta }(t,s,y)\right\vert  \\
		\leq &\bar{L}\left\vert \alpha -\beta \right\vert +L\left\vert
		x-y\right\vert .
	\end{align*}
	Applying Theorem \ref{BDG} and H\"{o}lder's inequality, we have 
	\begin{align*}
		&\hat{\mathbb{E}}[|X_{\alpha }(t)-X_{\beta }(t)|^{2}] \\
		\leq &C\left\{ \hat{\mathbb{E}}[\left\vert \phi _{\alpha }(t)-\phi _{\beta}(t)\right\vert ^{2}]+\hat{\mathbb{E}}\left[ \left\vert\int_{0}^{t}b_{\alpha }(t,s,X_{\alpha }(s))-b_{\beta }(t,s,X_{\beta }(s))\mathrm{d}s\right\vert ^{2}\right] \right.  \\
		&\left. +\hat{\mathbb{E}}\left[ \left\vert \int_{0}^{t}h_{\alpha
		}(t,s,X_{\alpha}(s))-h_{\beta}(t,s,X_{\beta}(s))\mathrm{d}\left\langle B\right\rangle(s)\right\vert^{2}\right]+\hat{\mathbb{E}}\left[ \left\vert\int_{0}^{t}\sigma _{\alpha }(t,s,X_{\alpha }(s))-\sigma _{\beta}(t,s,X_{\beta }(s))\mathrm{d}B(s)\right\vert ^{2}\right] \right\}  \\
		\leq &C_{\bar{\sigma},T}\left\{ \bar{L}^{2}\left\vert \alpha -\beta\right\vert ^{2}+\hat{\mathbb{E}}\left[ \left\vert \int_{0}^{t}L\left\vert X_{\alpha }(s)-X_{\beta}(s)\right\vert \mathrm{d}s\right\vert ^{2}\right] +\hat{\mathbb{E}}\left[
		\int_{0}^{t}\left\vert \sigma _{\alpha }(t,s,X_{\alpha }(s))-\sigma _{\beta}(t,s,X_{\beta }(s))\right\vert ^{2}\mathrm{d}s\right] \right\}  \\
		\leq &C_{\bar{\sigma},T,\bar{L}}\left\{ \left\vert \alpha -\beta\right\vert^{2}+\hat{\mathbb{E}}\left[\int_{0}^{t}L^{2}\left\vert X_{\alpha }(s)-X_{\beta }(s)\right\vert ^{2}\mathrm{d}s\right] +\hat{\mathbb{E}}\left[ \int_{0}^{t}\left\vert \bar{L}\left\vert \alpha -\beta \right\vert+L\left\vert X_{\alpha }(s)-X_{\beta }(s)\right\vert \right\vert ^{2}\mathrm{d}s\right] \right\}  \\
		\leq&C_{\bar{\sigma},T,L,\bar{L}}\left\{\left\vert\alpha-\beta\right\vert ^{2}+\hat{\mathbb{E}}\left[ \int_{0}^{t}\left\vert X_{\alpha}(s)-X_{\beta }(s)\right\vert ^{2}\mathrm{d}s\right] \right\}  \\
		\leq &C_{\bar{\sigma},T,L,\bar{L}}\left\{ \left\vert \alpha -\beta
		\right\vert ^{2}+\int_{0}^{t}\hat{\mathbb{E}}\left[ \left\vert X_{\alpha}(s)-X_{\beta }(s)\right\vert ^{2}\right] \mathrm{d}s\right\} .
	\end{align*}%
	Due to Gronwall's inequality, we yield 
	\begin{align*}
		\hat{\mathbb{E}}[|X_{\alpha }(t)-X_{\beta }(t)|^{2}] \leq &C_{\bar{\sigma},T,L,\bar{L}}\left\{ \left\vert \alpha -\beta \right\vert ^{2}+\int_{0}^{t}\hat{\mathbb{E}}\left[ \left\vert X_{\alpha }(s)-X_{\beta }(s)\right\vert^{2}\right] \mathrm{d}s\right\}  \\
		\leq &C_{\bar{\sigma},T,L,\bar{L}}\left\vert \alpha -\beta \right\vert ^{2}.
	\end{align*}
	From Theorem \ref{continuous}, we yield that there exist a $\delta $-order H\"{o}lder's continuous modification of $X_{\alpha }(t),\ \alpha \in \mathbb{R}$ for $\delta \in \left[ 0,\frac{1}{2}\right) $. The proof is complete.
\end{proof}

\section*{Declarations}

\subsection*{Funding}
Not applicable.
\subsection*{Ethical approval}
Not applicable.
\subsection*{Informed consent}
Not applicable.
\subsection*{Author Contributions}
All authors contributed equally to each part of this work. All authors read and approved the final manuscript.
\subsection*{Data Availability Statement}
Not applicable.
\subsection*{Conflict of Interest}
The authors declare that they have no known competing financial interests or personal relationships that could have appeared to influence the work reported in this paper.
\subsection*{Clinical Trial Number}
Not applicable.


\bigskip




\begin{thebibliography}{99} 
	
\bibitem{bai2014on}X. Bai, Y. Lin, On the exsitence and uniqueness of solution to stochastic differential equations driven by $G$-brownian motion with integral-lipschtiz coefficients, Acta Math. Appl. Sin. Engl. Ser. 30 (3) (2014) 589–610.
	
\bibitem{berger1980volterra}M. Berger, V. Mizel, Volterra equations with It\^{o} integrals, I, II, J. Integral Equ. 2 (1980) 187–245, 319–337.

\bibitem{chen2007a}S. Chen, J. Yong, A linear quadratic optimal control problems for stochastic Volterra integral equations, in Control Theory and Related Topics: In Memory of Prof. Xunjing Li, edited by S. Tang and J. Yong. Word Scientific Publishing Company (2007) 44–66.

\bibitem{denis2011function}L. Denis, M. Hu, S. Peng, Function spaces and capacity related to a sublinear expectation: application to $ G $-Brownian motion paths, Potential Anal. 34 (2011) 139–161.

\bibitem{ferreyra2000comparison}G. Ferreyra, P. Sundar, Comparison of stochastic Volterra equations, Bernoulli 6 (2000) 1001–1006.

\bibitem{gao2009pathwise}F. Gao, Pathwise properties and homomorphic flows for stochastic differential equations driven by $ G $-Brownian motion, Stochastic Process. Appl. 119 (10) (2009) 3356–3382.

\bibitem{hu2014backward}M. Hu, S. Ji, S. Peng, Y. Song, Backward stochastic differential equations driven by $G$-Brownian motion, Stoch. Process. Appl. 124 (2014) 759–784.

\bibitem{hu2009representation}M. Hu, S. Peng, On representation theorem of $ G $-expectations and paths of $ G $-Brownian motion, Acta Math. Appl. Sin. Engl. Ser. 25 (3) (2009) 539–546.

\bibitem{hu2016quasi}M. Hu, F. Wang, G. Zheng, Quasi-continuous random variables and processes under the $ G $-expectation framework, Stoch. Process. Appl. 126 (2016) 2367–2387.


\bibitem{ito1979on}I. Ito, On the existence and uniqueness of solutions of stochastic integral equations of the Volterra type, Kodui Math. J. 2 (1979) 158-170.

\bibitem{li2011stopping}X. Li, S. Peng, Stopping times and related It\^{o}'s calculus with $ G $-Brownian motion, Stochastic Process. Appl. 121 (2011) 1492–1508.

\bibitem{liu2019multi}G. Liu, Multi-dimensional BSDEs driven by $G$-Brownian motion and related system of fully nonlinear PDEs, Stochastics 92(5) (2019) 659–683. 

\bibitem{pardoux1990stochastic}E. Pardoux, P. Protter, Stochastic Volterra equations with anticipating coefficients. Ann. Probab. 18 (1990) 1635–1655.

\bibitem{peng2004filtration}S. Peng, Filtration consistent nonlinear expectations and evaluations of contingent claims, Acta Math. Appl. Sin. 20 (2004) 1–24.

\bibitem{peng2005nonlinear}S. Peng, Nonlinear expectations and nonlinear Markov chains, Chinese Ann. Math. 26B (2) (2005) 159–184.

\bibitem{peng2007G}S. Peng, $ G $-expectation, $ G $-Brownian motion and related stochastic calculus of It\^{o} type, in: Stochastic Analysis and Applications, in: Abel Symp., vol. 2, Springer, Berlin, 2007, pp. 541–567.


\bibitem{peng2008multi}S. Peng, Multi-dimensional $ G $-Brownian motion and related stochastic calculus under $ G $-expectation, Stochastic Process. Appl. 118 (12) (2008) 2223–2253.

\bibitem{peng2019nonlinear}S. Peng, Nonlinear Expectations and Stochastic Calculus under Uncertainty, Springer-Verlag, Heidelberg, 2019.

\bibitem{protter1985volterra}P. Protter, Volterra equations driven by semimartingales. Ann. Probab. 13 (1985) 519–530.


\bibitem{shi2015optimal}Y. Shi, T. Wang, J. Yong, Optimal control problems of forward-backward stochastic Volterra integral equations. Math. Control Relat. Fields 5 (2015) 613–649.

\bibitem{tudor1989a}C. Tudor, A comparion theorem for stochastic equations with Volterra drifts, Ann. Probab. 17 (1989) 1541–1545.

\bibitem{wang2018lq}T. Wang, Linear quadratic control problems of stochastic Volterra integral equations. ESAIM Control Optim. Calc. Var. 24(4) 1849–1879.

\bibitem{wang2010symmetrical}T. Wang, Y. Shi, Symmetrical solutions of backward stochastic Volterra integral equations and applications. Discrete Contin. Dyn. Syst. B. 14 (2010) 251–274.

\bibitem{wang2015com}
T. Wang, J. Yong, Comparison theorems for some backward stochastic Volterra integral equations, Stochastic Process. Appl. 125 (2015) 1756–1798.

\bibitem{wang2017optimal}T. Wang, H. Zhang, Optimal control problems of forward-backward stochastic Volterra integral equations with closed control regions. SIAM J. Control Optim. 55(4) (2017) 2574–2602.


\bibitem{wang2017parameter}Y. Wang, Stochastic Volterra integral equations with a parameter, Adv. Differ. Equ. 2017 (2017) 333.

\bibitem{wang2008existence}Z. Wang, Existence and uniqueness of solutions to stochastic Volterra equations with singular kernels and non-Lipschitz coefficients, Statist. Probab. Lett. 78 (2008) 1062–1071.

\bibitem{yong2006backward}J. Yong, Backward stochastic Volterra integral equations and some related problems. Stochastic Process Appl. 116 (2006) 779–795.

\bibitem{yong2008well}J. Yong, Well-posedness and regularity of backward stochastic Volterra integral equation. Probab. Theory Related Fields 142 (2008) 21–77.

\bibitem{zhang2010stochastic}X. Zhang, Stochastic Volterra equations in Banach spaces and stochastic partial differential equation. J. Funct. Anal. 258 (2010) 1361–1425.

\bibitem{zhao2025G-SVIE}B. Zhao, R. Li, M. Hu, Stochastic Volterra Integral Equations Driven by $G$-Brownian Motion, Math. Methods Appl. Sci. (2025).





	
	
	
	
	
\end{thebibliography}
\end{document}